\documentclass[11pt,a4paper]{amsart}
\usepackage{a4wide}
\usepackage[latin1]{inputenc}
\usepackage{amsmath}
\usepackage{amsfonts}
\usepackage{amssymb}
\usepackage{graphicx}
\usepackage{amsthm}
\usepackage{amssymb}
\usepackage{cite}
\usepackage{mathrsfs}
\usepackage{hyperref}
\usepackage{paralist}
\usepackage{enumitem}
\usepackage{tikz}
\usetikzlibrary{matrix,chains}
\usepackage{dynkin-diagrams}

\begin{document}

	\newcommand{\n}{\mathbf{n}}
	\newcommand{\x}{\mathbf{x}}
	\newcommand{\h}{\mathbf{h}}
	\newcommand{\m}{\mathbf{m}}
	\newcommand{\Ph}{\Phi_}
	
	\newcommand{\bN}{\mathbf{N}}
	\newcommand{\bL}{\mathbf{L}}
	\newcommand{\bG}{\mathbf{G}}
	\newcommand{\bS}{\mathbf{S}}
	\newcommand{\bZ}{\mathbf{Z}}
	\newcommand{\bH}{\mathbf{H}}
	\newcommand{\B}{\mathbf{B}}
	\newcommand{\U}{\mathbf{U}}
	\newcommand{\K}{\mathbf{K}}
	\newcommand{\V}{\mathbf{V}}
	\newcommand{\T}{\mathbf{T}}
	\newcommand{\G}{\mathbf{G}}
	\newcommand{\Para}{\mathbf{P}}
	\newcommand{\Levi}{\mathbf{L}}
	\newcommand{\Y}{\mathbf{Y}}
	\newcommand{\X}{\mathbf{X}}
	\newcommand{\M}{\mathbf{M}}
	\newcommand{\pro}{\mathbf{prod}}
	\renewcommand{\o}{\overline}
	
	\newcommand{\Gtilde}{\mathbf{\tilde{G}}}
	\newcommand{\Ttilde}{\mathbf{\tilde{T}}}
	\newcommand{\Btilde}{\mathbf{\tilde{B}}}
	\newcommand{\Ltilde}{\mathbf{\tilde{L}}}
	\newcommand{\C}{\operatorname{C}}

	\newcommand{\N}{\operatorname{N}}
	\newcommand{\bl}{\operatorname{bl}}
	\newcommand{\Z}{\operatorname{Z}}
	\newcommand{\Gal}{\operatorname{Gal}}
	\newcommand{\modulo}{\operatorname{mod}}
	\newcommand{\kernel}{\operatorname{ker}}
	\newcommand{\Irr}{\operatorname{Irr}}
	\newcommand{\D}{\operatorname{D}}
	\newcommand{\I}{\operatorname{I}}
	\newcommand{\GL}{\operatorname{GL}}
	\newcommand{\SL}{\operatorname{SL}}
	\newcommand{\W}{\mathbf{W}}
	\newcommand{\R}{\operatorname{R}}
	\newcommand{\Br}{\operatorname{Br}}
	\newcommand{\Aut}{\operatorname{Aut}}
	\newcommand{\End}{\operatorname{End}}
	\newcommand{\Ind}{\operatorname{Ind}}
	\newcommand{\Res}{\operatorname{Res}}
	\newcommand{\br}{\operatorname{br}}
	\newcommand{\Hom}{\operatorname{Hom}}
	\newcommand{\Endo}{\operatorname{End}}
	\newcommand{\Ho}{\operatorname{H}}
	\newcommand{\Tr}{\operatorname{Tr}}
	\newcommand{\opp}{\operatorname{opp}}
	\newcommand{\ssc}{\operatorname{sc}}
	\newcommand{\ad}{\operatorname{ad}}

	\newcommand{\tw}[1]{{}^#1\!}

	\newtheorem{definition}{Definition}[section]
	\newtheorem{notation}[definition]{Notation}
	\newtheorem{construction}[definition]{Construction}
	\newtheorem{remark}[definition]{Remark}
	\newtheorem{example}[definition]{Example}

	\newtheorem{theorem}[definition]{Theorem}
	\newtheorem{lemma}[definition]{Lemma}
	\newtheorem{question}{Question}
	\newtheorem{corollary}[definition]{Corollary}
	\newtheorem{proposition}[definition]{Proposition}
	\newtheorem{conjecture}[definition]{Conjecture}
	\newtheorem{assumption}[definition]{Assumption}
	\newtheorem{hypothesis}[definition]{Hypothesis}
	\newtheorem{maintheorem}[definition]{Main Theorem}

	\newtheorem{theo}{Theorem}
	\newtheorem{conj}[theo]{Conjecture}
	\newtheorem{cor}[theo]{Corollary}

	\renewcommand{\thetheo}{\Alph{theo}}
	\renewcommand{\theconj}{\Alph{conj}}
	\renewcommand{\thecor}{\Alph{cor}}

	\newcommand{\ov}{\overline }
	
	\def\oo#1{\overline{\overline{#1}}}

	\title{Equivariant chararacter bijections and the inductive Alperin--McKay condition}

	\date{\today}
	\author{Julian Brough}
	\author{Lucas Ruhstorfer}
		\address{J.B.}
	\email{julian.m.a.brough@gmail.com}
	
	\address{L.R. School of Mathematics and Natural Sciences University of Wuppertal, Gaußstr. 20,
	42119 Wuppertal, Germany}
\email{ruhstorfer@uni-wuppertal.de}

	
	\begin{abstract}
		In this paper we consider the inductive Alperin--McKay condition for isolated blocks of groups of Lie type $B$ and $C$. This finishes the verification of the inductive condition for groups of this type.
		%The Alperin--McKay conjecture is a longstanding open conjecture in the representation theory of finite groups.
		%	Späth showed that the Alperin--McKay conjecture holds if the so-called inductive Alperin--McKay (iAM) condition holds for all finite simple groups. In a previous paper, the author has proved that it is enough to verify the inductive condition for quasi-isolated blocks of groups of Lie type. In this paper we show that the verification of the iAM-condition can be further reduced in many cases to isolated blocks. As a consequence of this we obtain a proof of the Alperin--McKay conjecture for $2$-blocks of finite groups with abelian defect.
	\end{abstract}
	
	\maketitle

	\maketitle

	In the representation theory of finite groups some of the most important open conjectures relate the representation theory of a finite group $G$ and those of its $\ell$-local subgroups, for $\ell$ a prime dividing the order of $G$.
	One of these conjectures is the Alperin--McKay conjecture, which forms a blockwise generalisation of the McKay conjecture.
	For an $\ell$-block $b$ of $G$ we denote by ${\rm Irr}_0(G,b)$ the subset of height zero characters of $\Irr(G,b)$, the characters of the block $b$.
	
	\begin{conj}[Alperin--McKay]
		Let $b$ be an $\ell$-block of $G$ with defect group $D$ and $B$ its Brauer correspondent in ${\rm N}_G(D)$. 
		Then $$|{\rm Irr}_0(G,b)|=|{\rm Irr}_0({\rm N}_G(D),B)|.$$
	\end{conj}
	
	In \cite{IAM}, the  Alperin--McKay conjecture was reduced to the verification of the so-called {\it inductive Alperin--McKay condition} (iAM) for all finite simple groups and primes $\ell$.
	In previous papers \cite{Jordan3} for simple groups of Lie type, the second author has proven that it suffices to verify the iAM-condition for quasi-isolated blocks of groups of Lie type and in many cases just for isolated blocks.
	Applying these techniques together with results from \cite{Brough}, yielded a verification of the iAM-condition for all finite simple groups of type $A$ and primes $\ell\geq 5$.
	
	The present paper is concerned with considering the iAM-condition for the remaining classical quasi-simple groups, i.e., for groups of Lie type $B$, $C$ and $D$ defined over a field of characteristic $p \neq \ell$.
	
	Due to the nature of the inductive proof using the iAM-condition, there is some flexibility in the choice of local subgroups that can be taken.
	When $\ell\geq 5$ the defect groups of $\ell$-blocks of these groups have a Cabanes subgroup, i.e. they have a unique maximal normal abelian subgroup.

	With this in mind the first half of this paper is focused on the following result, which provides a method to verify the iAM-condition.

	\begin{theo}\label{RedCabQIso}
		Assume that $\G$ is a simple algebraic group of simply connected type $B_n$ or $C_n$ ($n \geq 2$) with Frobenius endomorphism $F: \G \to\G$ and let $b$ be a quasi-isolated $\ell$-block of $G:=\G^F$ with $\ell \geq 5$.
%		Let $\mathcal{B}$ denote the subgroup of ${\rm Aut}(G)$ generated by the graph and field automorphisms as in Section \ref{LocalChar}, $\ell \geq 5$, and 
		Assume that the block $b$ satisfies Assumption \ref{assumption}.
%		If the stabilizer $\mathcal{B}_b$ is cyclic,
		Then the block $b$ is AM-good relative to the Cabanes subgroup of its defect group.
	\end{theo}
	
%		Note the condition on $\mathcal{B}_b$ will hold for all blocks in types $B_n$ and $C_n$ (with $n\geq 3$) as $\mathcal{B}$ is cyclic for those groups. Although Theorem \ref{RedCabQIso} still requires $\mathcal{B}_b$ to be cyclic, which is not always satisfied in untwisted type $D$, the above criterion seems less restrictive than the one in \cite[Theorem 2.4]{Brough}, see Remark \ref{rem}.

%	The group $G$ in the previous theorem occurs as the fixed point subgroup of an algebraic group $\G$ under a suitable Frobenius endomorphism $F$.
	The block $b$ of $G$ is labeled by a pair $(\Levi,\lambda)$ consisting of a $d$-split Levi subgroup $\Levi$ of $(\G,F)$ and a $d$-cuspidal character $\lambda\in {\rm Irr}(\Levi^F)$, where $d$ is the order of $q$ modulo $\ell$, see \ref{parameter}.
	Assumption \ref{assumption} concerns the Clifford theory associated to $\Levi^F \lhd \N_G(\Levi)$ taking into account the action arising from the automorphisms of $G$ stabilizing $\Levi$.
	
	Assumption \ref{assumption} was verified in type $C$ with respect to any $d$-cuspidal pair \cite{BroughC} and for blocks of maximal defect in type $B$ \cite{TypeB}. In types $B$ and $C$, by \cite[Theorem D]{Jordan3} it suffices to validate the iAM-condition for all isolated $\ell$-blocks.
	The second half of this paper therefore focuses on verifying this assumption also for the isolated $\ell$-blocks in type $B_n$. In contrast to the computation in \cite{BroughC} we make explicit use of the structure of the possible Levi subgroups that can occur for isolated blocks. This for one reduces the computational effort of verifying Assumption \ref{assumption} and is helpful in overcoming technical difficulties related to the more complicated structure of the extended Weyl group (see e.g. \cite[Remark 3.4]{BroughC}). By verifying Assumption \ref{assumption} we can therefore finish the verification of the iAM-condition for the considered groups.
	
	\begin{theo}
		Let $G$ be a quasi-simple group of Lie type $B_n$ or $C_n$, with $n\geq 2$, defined over the finite field $\mathbb{F}_q$ for $q$ a prime power of an odd prime and let $\ell \geq 5$ not dividing $q$.
		Then every $\ell$-block of $G$ satisfies the iAM-condition.
	\end{theo}
	
\subsection*{Structure of the paper.}
	In Section~\ref{Notation} we provide some of the basic notation that will be used throughout the paper.
	Sections \ref{GlobChar} and \ref{LocalChar} are dedicated to the parametrisation of the global, respectively local, height zero characters via the techniques developed by Enguehard together with $d$-Harish-Chandra theory.
	This will then be used in Section~\ref{discentre} to prove Theorem~\ref{RedCabQIso}.
%	Section~\ref{TypeB} is then dedicated to
The remaining sections are then dedicated to verifying Assumption \ref{assumption} in types $B_n$.

\subsection*{Acknowledgment}

The research of the first author is funded through the DFG (Project: BR 6142/1-1). The second author thanks Jay Taylor for helpful conversations. Moreover, he would like to thank the Isaac Newton Institute for Mathematical Sciences for
support and hospitality during the programme "Groups, representations and applications:
new perspectives" when work on this paper was undertaken. This work was supported by:
EPSRC grant number EP/R014604/1. This paper is a contribution via the second author to SFB TRR 195.

The authors thank Britta Späth and Gunter Malle for carefully reading a previous version of this paper.
\section{notation}\label{Notation}

%{\color{purple}
	\subsection{Finite groups of Lie type}
	
	Throughout, ${\bf G}$ denotes a connected reductive group over an algebraic closure $\mathbb{F}$ of $\mathbb{F}_p$ for a prime number $p$.	Let $F:\G\rightarrow \G$ be a Frobenius endomorphism defining an $\mathbb{F}_q$-structure on $\G$ where $q$ is an integral power of $p$. We let $(\G^\ast,F)$ be a pair in duality with $(\G,F)$. If $\mathbf{H}$ is an $F$-stable subgroup of $\G$ (or of $\G^\ast$), then we denote $H:=\mathbf{H}^F$. 

%	\subsection{$d$-split Levi subgroups}
	An $F$-stable torus $\T$ of $\bG$ is called an \emph{$e$-torus} if it splits
	completely over $\mathbb{F}_{q^e}$ but no non-trivial subtorus splits over any smaller field. In particular, $|\T^F|$ is a power of $\Phi_e(q)$, where $\Phi_e$ denotes the $e$th cyclotomic polynomial. For an $F$-stable torus $\T$ we denote by $\T_{\phi_e}$ its maximal $e$-split subtorus.
%	Equivalently, there is $a\ge0$ such that $|\bT^{F^k}|=\Phi_e(q^k)^a$ for
%	all $k\ge1$, where $\Phi_e$ denotes the $e$th cyclotomic polynomial.
	The
	centralizers of $e$-tori of $\bG$ are called \emph{$e$-split Levi subgroups}.
%	(Note that these are indeed Levi subgroups, which are $F$-stable.)
	
	\subsection{Character theory}
	For any $F$-stable Levi subgroup $\bL$ of
	a (not necessarily $F$-stable) parabolic subgroup $\Para$ of $\bG$ Lusztig
	defines linear maps
	$$R_{\Levi \subset \Para}^\G:\mathbb{Z}\Irr(L)\longrightarrow\mathbb{Z}\Irr(G),$$
	$${}^\ast R_{\Levi \subset \Para}^\G:\mathbb{Z}\Irr(G)\longrightarrow\mathbb{Z}\Irr(L).$$
%	adjoint to each other with respect to the standard scalar
%	product on complex characters.
	By the Mackey formula (see \cite[Theorem 3.3.7]{GeckMalle}) these maps won't depend in our cases on the choice of parabolic subgroups so that we will write ${}^{(\ast)} R_\Levi^\G$ instead of ${}^{(\ast)}R_{\Levi \subset \Para}^\G$.
	
	A character $\chi\in\Irr(G)$ is called \emph{$e$-cuspidal} if
	${}^\ast R_{\Levi}^\G(\chi)=0$ for every $e$-split proper Levi subgroup $\bL$ of $\bG$. A pair $(\bL,\lambda)$ consisting of an $e$-split Levi subgroup $\bL$ of $\G$ and an
	$e$-cuspidal character $\lambda\in\Irr(L)$ is then called an
	\emph{$e$-cuspidal pair}.
	Given an $e$-cuspidal pair $(\bL,\lambda)$, we write
	$$\mathcal{E}(G,(\bL,\lambda)):=\{\chi\in\Irr(G)\mid
	\langle {}^\ast R_\Levi^{\G}(\chi),\lambda\rangle\ne0\}$$
	for the set of constituents of $R_\Levi^{\G}(\lambda)$. This is called the
	\emph{$e$-Harish-Chandra series of $G$ above $(\bL,\lambda)$}.
	
	For a semisimple element $s \in G^\ast$ we denote by $\mathcal{E}(G,s)$ the rational Lusztig series associated to it. For ${z} \in \Z(G^\ast)$ there exists a linear character $\hat{z} \in \Irr(G)$ and multiplication with $\hat{z}$ induces a bijection $\hat{z}: \mathcal{E}(G,1) \to \mathcal{E}(G,z)$, see \cite[Equation 8.19]{MarcBook}.
	
	Moreover, if the center $\Z(\G)$ is connected
	$$\psi_{G,s}: \mathcal{E}(\C_{G^\ast}(s),1) \to \mathcal{E}(G,s)$$
	will denote the unique Jordan decomposition from Digne--Michel \cite[Theorem 4.7.1]{GeckMalle}. In particular, we have a bijection $\psi_{G,1}:\mathcal{E}(G^\ast,1) \to \mathcal{E}(G,1)$ between the set of unipotent characters of $G$ and $G^\ast$.

\subsection{Blocks of groups of Lie type}\label{block lie}

We let $\ell \neq p$ be a fixed prime and denote by $\mathcal{E}(G,\ell')$ the union of Lusztig series $\mathcal{E}(G,s)$ with $s \in G^\ast$ a semisimple element of $\ell'$-order. For such a semisimple element $s$ of $\ell'$-order, we denote by $\C_{G^\ast}(s)_\ell$ the subset of $\ell$-power order elements of the centralizer of $s$. By a fundamental result of Broué--Michel \cite[Theorem 9.12]{MarcBook}, the set
$$\mathcal{E}_\ell(G,s)=\displaystyle\bigcup_{t \in \C_{G^\ast}(s)_\ell} \mathcal{E}(G,st)$$
is the set of characters associated to a sum of $\ell$-blocks. We denote by $\mathrm{Bl}(G,s)$ the corresponding set of blocks.

%}

\section{Global characters}\label{GlobChar}

\subsection{Parametrizing height zero characters}\label{parameter}

Throughout this section, we let $\G$ be a connected reductive group with connected centre with all simple components of classical type. Suppose that $\ell$ is a prime with $\ell \nmid q$ and $\ell \geq 5$. We denote by $d$ the order of $q$ modulo $\ell$. According to \cite[Theorem 4.1]{Marc} to each $d$-cuspidal pair $(\Levi,\lambda)$ with $\lambda \in \mathcal{E}(L,\ell')$, up to $\G^F$-conjugation, there exists a unique $\ell$-block $b=b_{\G^F}(\Levi,\lambda)$ of $\G^F$ associated to it. Let $s \in L^\ast$ a semisimple element of $\ell'$-order such that $\lambda \in \mathcal{E}(L,s)$. Observe that since $(\Levi,\lambda)$ is $d$-cuspidal, we have $\Z^\circ(\C^\circ_{\Levi^\ast}(s))_{\Phi_d}=\Z^\circ(\Levi^\ast)_{\Phi_d}$, see \cite[Remark 2.2]{Cabanesgroup}. Since $\Levi^\ast$ is $d$-split we have $\Levi^\ast=\C_{\G^\ast}(\Z^\circ(\Levi^\ast)_{\Phi_d})$ and so $\C_{\Levi^\ast}(s)=\C_{\C_{\G^\ast}(s)}(\Z^\circ(\C^\circ_{\Levi^\ast}(s))_{\Phi_d})$ is a $d$-split Levi subgroup of $(\C_{\G^\ast}(s),F)$.

\begin{lemma}\label{lemma1}
	With the assumptions and notation from above,
%	
%let $b$ be the $\ell$-block of $\G^F$ associated to the $d$-cuspidal pair $(\Levi,\lambda)$ with $\lambda \in \mathcal{E}(L,s)$ as above. Then
we have
$$\Irr_0(G,b) \subseteq \displaystyle\bigcup_{t \in \Z(\C_{L^\ast}(s))_\ell} \mathcal{E}(G,st).$$
%	Let $t \in \mathrm{C}_{G^\ast}(s)_\ell$ and assume that $\chi \in \Irr(G,b) \cap \mathcal{E}(G,st)$ has height zero. Then $t \in \mathrm{Z}(\mathrm{C}_{\Levi^ \ast}(s))_\ell$, i.e. $\mathrm{C}_{L^\ast}(st)=\mathrm{C}_{L^\ast}(s)$.
\end{lemma}

\begin{proof}
	Let $\G(s)$ be a connected reductive group in duality with $\C_{\G^\ast}(s)$ (this group is connected by \cite[Proposition 13.16]{MarcBook} as $\Z(\G)$ is connected by assumption). According to \cite[Theorem 1.6]{EnguehardJordandecomp}, there exists a  bijection
	$$\mathcal{B}_{G,s}: \mathrm{Bl}(\G(s)^F,1) \to \mathrm{Bl}(\G^F,s)$$
	such that for $t \in  \mathrm{C}_{G^\ast}(s)_\ell$ we have
	$$\psi_{G,st} \circ \psi_{G(s),t}^{-1}( \Irr(G(s),c) \cap \mathcal{E}(G(s),t))= \Irr(G,b) \cap \mathcal{E}(G,st),$$
where $\mathcal{B}_{G,s}(c)=b$.
By \cite[Theorem 1.6]{EnguehardJordandecomp} this bijection yields a height preserving bijection $\Irr(G(s),c) \to \Irr(G,b)$.
The block $c$ is determined from $b$ by \cite[Theorem 1.6(B.1.a)]{EnguehardJordandecomp} as follows: By Jordan decomposition $\lambda$ corresponds to a unipotent character $\lambda_s \in \mathcal{E}(\C_{L^\ast}(s),1)$. Therefore, there is a Levi subgroup $\Levi(s)$ of $\G(s)$ in duality with the $d$-split Levi subgroup $\C_{\Levi^\ast}(s)$ of $\C_{\G^\ast}(s)$. We let $\lambda(s):= \psi_{L(s),1}(\lambda_s)$. Then $c$ is the block associated to the unipotent $d$-cuspidal pair $(L(s),\lambda(s))$ of $G(s)$.

By \cite[Theorem 5.2]{Brough} we therefore deduce that 
$$\Irr_0(G(s),c) \subseteq \displaystyle\bigcup_{t \in \Z(L(s)^\ast)_\ell} \mathcal{E}(G(s),st).$$
By definition $L(s)^\ast=\C_{L^\ast}(s)$, so $t \in \Z(\C_{L^\ast}(s))$. Using that the height preserving bijection $\Irr(G(s),c) \to \Irr(G,b)$ preserves Lusztig series yields the claim. 
\end{proof}

\begin{lemma}\label{ht zero}
Keep the notation and assumptions of Lemma \ref{lemma1} and let $\chi \in \mathcal{E}(G,st) \cap \Irr_0(G,b)$ be a height zero character. Then $\chi$ is a constituent of $R_{\Levi}^{\G}(\theta)$, where $\theta$ is the $d$-cuspidal character defined by $\theta:=\psi_{L,st} (\psi_{L,s}^{-1}(\lambda)) \in \mathcal{E}(L,st)$. Note that by Lemma \ref{lemma1} the character $\theta$ is well-defined as $\mathrm{C}_{L^\ast}(st)=\mathrm{C}_{L^\ast}(s)$.
\end{lemma}

\begin{proof}
%We keep the notation of Lemma \ref{lemma1}.
Let $\mathbf{K}$ be a Levi subgroup of $\G$ in duality with the Levi subgroup $\C_{\G^\ast}(t)$ of $\G^\ast$. Note that the latter is indeed a Levi subgroup of $\G^\ast$ since $\ell \geq 5$ (see \cite[Theorem 13.16]{MarcBook}).
 By \cite[Proposition 2.2.4]{EnguehardJordandecomp} there exists a $d$-cuspidal pair $(\Levi_\K,\lambda_\K)$ of $\K$ such that the map $R_{\mathbf{K}}^{\G} \hat{t}$ yields a bijection
$$\Irr(b_{K}(\Levi_\K,\lambda_\K)) \cap \mathcal{E}(K,s) \to \Irr(b_{G}(\Levi,\lambda)) \cap \mathcal{E}(G,st).$$
Here, the $d$-cuspidal pair $(\Levi_\K,\lambda_\K)$ of $\K$ is constructed as follows: Let $(\C_{\Levi^\ast}(s),\lambda_s)$ be the unipotent $d$-cuspidal pair of $\C_{\G^\ast}(s)$ corresponding under Jordan decomposition to the $d$-cuspidal pair $(\Levi,\lambda)$ of $G$. By \cite[Definition and Proposition 1.4.7]{EnguehardJordandecomp} there exists a unipotent $d$-cuspidal pair $(\C_{\Levi^\ast}(s)_{\C_{\K^\ast}(s)},(\lambda_s)_{\C_{\K^\ast}(s)})$ of $\C_{\K^\ast}(s)$ such that $[\C_{\Levi^\ast}(s),\C_{\Levi^\ast}(s)]=[\C_{\Levi^\ast}(s)_{\C_{\K^\ast}(s)},\C_{\Levi^\ast}(s)_{\C_{\K^\ast}(s)}]^g$ for some $g \in \C_{G^\ast}(s)$, and the characters $\lambda_s$ and $(\lambda_s)_{\C_{\K^\ast}(s)}^g$ both restrict to the same character of $[\C_{\Levi^\ast}(s),\C_{\Levi^\ast}(s)]^F$. Then $(\Levi_{\K},\lambda_\K)$ is defined as the $d$-cuspidal pair of $\K$ with $\lambda_K \in \mathcal{E}(\Levi_{\K}^F,s)$ corresponding to $(\C_{\Levi^\ast}(s)_{\C_{\K^\ast}(s)},(\lambda_s)_{\C_{\K^\ast}(s)})$ under Jordan decomposition.

Observe that by Lemma \ref{lemma1} we have $t \in \Z(\C_{L^\ast}(s))$ and so $\C_{\G^\ast}(t) \cap \C_{\Levi^\ast}(s)=\C_{\Levi^\ast}(s)$. We therefore have $\K(s)^\ast=\Levi(s)^\ast$ and so $(\Levi(s)_{\K(s)^\ast},(\lambda_s)_{\K(s)^\ast})=(\Levi(s)^\ast,\lambda_s)$. Thus, $\Levi_{\K}$ is the Levi subgroup of $\K$ in duality with the Levi subgroup $\C_{\Levi^\ast}(t)$ of $\C_{\G^\ast}(t)$ and $\lambda_\K=\psi_{K,s}(\lambda_s) \in \mathcal{E}(K,s)$. Moreover, $\Levi_{\K}$ is a Levi subgroup of $\Levi$.
%Denote $(\lambda_K)_s:=\lambda_s$, where $\psi_{L,s}(\lambda_s)=\lambda$ is the Jordan correspondent of $\lambda \in \mathcal{E}(L,s)$.  we have $\mathrm{C}_{L^\ast}(st)=\mathrm{C}_{L^\ast}(s)$. It follows that 

Now the center of the Levi subgroup $\mathbf{K}$ is again connected by \cite[Proposition 13.12]{MarcBook}. By \cite[Proposition 5.2]{Marc} (and its proof) we find that any $\psi \in \Irr(b_{K}(\Levi_\K,\lambda_\K)) \cap \mathcal{E}(K,s)$ appears in the $d$-Harish-Chandra series of $(\Levi_\K,\lambda_\K)$. Hence, $\psi$ is a constituent of $R_{\Levi_\K}^\K(\lambda_\K)$ and thus $\chi:=R_\K^\G (\hat{t} \psi)$ is a constituent of $$R_\K^\G(\hat{t} R_{\Levi_\K}^\K(\lambda_\K))=R_{\Levi_\K}^\G(\lambda_\K \hat{t})=R_\Levi^\G(R_{\Levi_\K}^\Levi(\lambda_\K \hat{t})).$$ Now $ R^\Levi_{\Levi_\K} \circ \psi_{\Levi_\K,st}= \psi_{\Levi,st}$ by the properties of Jordan decomposition (see \cite[Theorem 4.7.1]{GeckMalle}). By what we said before, $st$ is central in $\C_{L^\ast}(s)$. Hence, $\chi$ is a constituent of $R_\Levi^\G(\theta)$ with $\theta:=R_{\Levi_\K}^\Levi(\lambda_\K \hat{t})=\psi_{L,st} \psi_{L,s}^{-1}(\lambda)$. Finally, observe that $\lambda_\K$ is a $d$-cuspidal character. Thus, $\lambda_\K \hat{t}$ is also $d$-cuspidal and hence by the proof of \cite[Proposition 4.1]{Cabanesgroup} the character $\theta$ is $d$-cuspidal as well.
\end{proof}

Note that the proof of the previous lemma shows in particular that all irreducible constituents of $R_{\Levi}^{\G}(\theta)$ lie in the block $b$. A similar description of the (height zero) characters of blocks of abelian defect using similar arguments was already obtained in \cite[Theorem 2.9]{MalleAlperin}.

\subsection{$d$-Harish-Chandra theory}
The previous lemma can now be used as a parametrization of the global height zero characters in terms of $d$-Harish-Chandra theory. Recall that if $[\G,\G]$ is simply connected, we have by \cite[Equation 8.19]{MarcBook} a bijection $$\Z(G^\ast) \to \Irr(\G^F/[\G,\G]^F), z \mapsto \hat{z}.$$

%As a consequence we obtain:

\begin{proposition}\label{dHC}
In addition to the assumptions from before we assume that $[\G,\G]$ is simply connected. Let $(\Levi,\theta)$, $\theta \in \mathcal{E}(L,st)$, be a $d$-cuspidal pair as in Lemma \ref{ht zero}. Let $\mathrm{Aut}_{\mathbb{F}}(\G^F)$ be defined as in \cite[2.4]{Spaeth}.
%For every $d$-cuspidal character $\theta \in \mathcal{E}(L,st)$,
There exist an $(\Irr(\G^F/[\G,\G]^F) \rtimes \mathrm{Aut}_{\mathbb{F}}(\G^F))_{(\Levi,\theta)}$-equivariant bijection
$$\Irr(W_{G}(\Levi, \theta) ) \to \mathcal{E}(\G^F,(\Levi,\theta)) ,\eta \mapsto R_\Levi^\G(\theta)_\eta.$$
Here, $(\hat{z},\sigma) \in (\Irr(\G^F/[\G,\G]^F) \rtimes \mathrm{Aut}_{\mathbb{F}}(\G^F))_{(\Levi,\theta)}$ acts by conjugation with $\sigma$ on $\Irr(W_{G}(\Levi, \theta) )$.
In particular, we have $R_\Levi^\G(\theta)_\eta \in \Irr_0(G,b)$ if and only if the following conditions are satisfied:
\begin{enumerate}[label=(\roman*)]
	\item $W_G(\Levi,\theta)$ contains a Sylow $\ell$-subgroup of $W_G(\Levi,\lambda)$.
	\item $\eta \in \Irr(W_G(\Levi,\theta))$ is an $\ell'$-character.
\end{enumerate}

\end{proposition}

\begin{proof}
	We follow the general ideas of \cite[Section 5.B]{Brough}.
Let $(\C_{\Levi^\ast}(st),\theta_{st})$ be the unipotent $d$-cuspidal pair in $\C_{\Levi^\ast}(st)$ associated to $(\Levi,\theta)$.
Then as observed in the proof of \cite[Proposition 5.4]{Brough} we have as in \cite[Theorem 3.4]{Spaeth} an $\mathrm{Aut}_{\mathbb{F}}(\C_{\G^\ast}(st))_{(\C_{\Levi^\ast}(st),\theta_{st})}$-equivariant bijection
$$I_{\C_{\Levi^\ast}(st),\theta_{st}}^{\C_{\G^\ast}(st)} : \Irr(W_{\C_{G^\ast}(s)}(\C_{\Levi^\ast}(st), \theta_{st}) ) \to \mathcal{E}(\C_{\G^\ast}(st), (\C_{\Levi^\ast}(st), \theta_{st})).$$
As in the proof of \cite[Lemma 5.3]{Brough} we obtain for $\eta \in \Irr(W_{\C_{G^\ast}(s)}(\C_{\Levi^\ast}(st), \theta_{st}) )$ the equality $$I_{\C_{\Levi^\ast}(st),\theta_{st}}^{\C_{\G^\ast}(st)} (\eta)(1)_\ell=|\C_{G^\ast}(st):\N_{\C_{G^\ast}(st)}(\C_{\Levi^\ast}(st))|_\ell \theta_{st}(1)_\ell \eta(1)_\ell.$$
By the uniqueness properties of Jordan decomposition (see also the proof of \cite[Theorem 2.9]{MalleAlperin}) duality induces a natural isomorphism
$$\Phi:W_{\C_{G^\ast}(st)}(\C_{\Levi^\ast}(st),\theta_{st}) \to W_{G}(\Levi,\theta).$$
In particular, we also obtain in this case a bijection
$$\Irr(W_{G}(\Levi, \theta) ) \to \mathcal{E}(G , (\Levi, \theta)), \, \eta \mapsto \psi_{G,st} \circ I_{\C_{\Levi^\ast}(st),\theta_{st}}^{\C_{\G^\ast}(st)} \circ (\eta \circ \Phi),$$
and as in the Harish-Chandra case we denote by $R_{\Levi}^\G(\theta)_\eta$ the character corresponding to $\eta \in \Irr(W_G(\Levi,\theta))$ under this bijection. As in \cite[Lemma 5.3]{Brough} the degree formula for the unipotent $d$-Harish-Chandra case yields that 
$$R_{\Levi}^{\G}(\theta)_\eta(1)_\ell=|G : \N _{C_{G^\ast}(st)}(\C_{\Levi^\ast}(st))|_\ell \theta_{st}(1)_\ell \eta(1)_\ell.$$
By Jordan decomposition, we have $\theta(1)_\ell=|L:\C_{L^\ast}(st)|_\ell \theta_{st}(1)_\ell$. Replacing this in the formula above yields
$$R_{\Levi}^{\G}(\theta)_\eta(1)_\ell=\frac{|G:L|_\ell}{|W_{\C_{\G^\ast}(st)}(\C_{\Levi^\ast}(st))|_\ell} \theta(1)_\ell \eta(1)_\ell.$$
As in the proof of \cite[Theorem 5.6]{Brough}, we deduce that the minimum over all such $(t,\eta)$ is obtained for $(1,1)$; hence the block $b$ has defect
$$  \frac{|\C_{G^\ast}(s)|_\ell}{\lambda_s(1)_\ell}.$$
It remains to show that the so-obtained bijection is $(\Irr(\G^F/[\G,\G]^F) \rtimes \mathrm{Aut}_{\mathbb{F}}(\G^F))_{(\Levi,\theta)}$-equivariant. For this, let $\sigma: \G \to \G$ be a bijective morphism commuting with the action of $F$. If $\sigma(\Levi)=\Levi$, then there exists a dual morphism $\sigma^\ast: \G^\ast \to \G^\ast$ commuting with $F^\ast$ and such that $\sigma|_{\Levi}$ and $\sigma^\ast|_{\Levi^\ast}$ are in duality.
%This uniquely determines $\sigma^\ast$, up to multiplication by $L^\ast \langle F \rangle$.
If $\hat{z} \sigma$ stabilizes $(\Levi,\theta)$ for some $z \in \Z(G^\ast)$, then $z \sigma^\ast(st)$ is $L^\ast$-conjugate to $st$. Hence, we can assume (by possibly replacing $\sigma^\ast$ by $l \sigma^\ast$ for some suitable $l \in L^\ast$) that $z \sigma^\ast(st)=st$. By the equivariance of Jordan decomposition \cite[Theorem 3.1]{Spaeth} it follows that $\sigma^\ast$ stabilizes $(\C_{\Levi^\ast}(st),\theta_{st})$. Using this construction, it follows that $\Phi:W_{\C_{G^\ast}(st)}(\C_{\Levi^\ast}(st),\theta_{st}) \to W_{G}(\Levi,\theta)$ is $(\sigma,\sigma^\ast)$-equivariant. The equivariance of Jordan decomposition (see \cite[Theorem 3.1]{Spaeth}) and the equivariance properties of the parametrization of unipotent $d$-Harish-Chandra series recalled at the beginning of the proof show that the constructed bijection is $(\Irr(\G^F/[\G,\G]^F) \rtimes \mathrm{Aut}_{\mathbb{F}}(\G^F))_{(\Levi,\theta)}$-equivariant.
\end{proof}
%We must also say something about defect groups and the associated $e$-split Levi subgroup (compare \cite[Theorem 5.1]{Brough}):
%
%\begin{lemma}\label{Cabanes}
%Let $(\Levi,\lambda)$ be a $e$-cuspidal pair of $(\G,F)$ associated to $b$. Let $S$ be the Sylow $\Phi_e$-subtorus of $\Z^\circ(\Levi)$. Then there exists some defect group $D$
%of $b$ such that $\N_{[[G,G]F} (S)$ is $Aut([G,G]^F)_{b,D}$-stable and $N_{G^F}(D) \leq \N_{G^F}(S)$. Moreover
%$C _{G^F}(D) \leq  \C _{G^F}(S) = \Levi^F$.
%\end{lemma}
%
%\begin{proof}
%Same as in Brough--Späth (?).
%\end{proof}
%
%It will be shown later that $S_\ell$ is the Cabanes subgroup of the defect group $D$.

\section{Local characters}\label{LocalChar}

From now on we assume that $\G$ is a simple, simply connected algebraic group of classical type $B_n$, $C_n$ or $D_n$. In particular, the center of $\G$ can be disconnected. Let $\iota: \G \hookrightarrow \widetilde{\G}$ be a regular embedding, i.e. a closed embedding of algebraic groups with $\Z(\widetilde{\bf G})$ connected and $\widetilde{\bf G}=\Z(\widetilde{\bf G}) {\bf G}$, see \cite[Section 15.1]{MarcBook}.
Assume $F:\widetilde{\G}\rightarrow \widetilde{\G}$ is a Frobenius endomorphism extending the one of $\G$, as in  \cite[Section 2]{MS}. 
%Via the conventions given in {\color{red} [MS16,2]}, $E(\G^F)$ also acts on $\widetilde{\G}^F$.

%Recall the following theorem by Malle--Navarro, Thm 7.7, Nilpotent blocks:
%
%\begin{theorem}\label{MalleNavarro}
%Let $\G$ be a connected reductive group with a Frobenius map $F : \G \to \G$. Let $B$ be a unipotent $\ell$-block of $G$. Then either $B$ is of central defect, and all characters
%of $B$ have the same degree, or $B$ contains two height 0 characters of different degrees.
%%Moreover, these two degrees have different r-parts unless possibly if r = 2.
%\end{theorem}

%Digne--Michel, Groupes non connexes yields:
%
%\begin{lemma}\label{bijection}
%	Assume that $[G,G]$ is simply connected so that centralizers of semisimple elements are connected. The map $\Levi \mapsto \C_{\Levi}(s)$ is a bijection between Levi subgroups of $\G$ containing $s$ and Levi subgroups of $\C_{G}(s)$. Its inverse is given by mapping a Levi subgroup $M$ of $\C_G(s)$ to $\C_{\G}(Z(M)^\circ)$.
%\end{lemma}
Recall that $\Phi_e$ denotes the $e$th cyclotomic polynomial. Set $E_{q,\ell}:=\{ e : \ell \mid \Phi_e(q)\}$ and observe that $E_{q,\ell}=\{d \ell^i \mid i \geq 1\}$ by the remarks before \cite[Lemma 13.17]{MarcBook}, where $d$ is the order of $q$ modulo $\ell$. We have the following elementary lemma:

\begin{lemma}\label{cyclotomic}
 Let $k$ be a positve integer with $(k,\ell)=1$. Then $(q^k - 1)_\ell \leq \Phi_d(q)_\ell$ (resp. $(q^k + 1)_\ell \leq \Phi_d(q)_\ell$) with equality if and only if $d \mid k$ (resp. $d \mid 2k$ but $d \nmid k$). 
\end{lemma}

\begin{proof}
	We write $q^{k}-1=\prod_{e \mid k} \Phi_e(q)$ resp. $q^{k}+1 = \prod_{e \mid 2k, e \nmid k} \Phi_e(q)$. The property follows by the above characterization of $E_{q,\ell}$.
\end{proof}

\begin{remark}\label{scalar descent}
	Recall \cite[Example 3.5.20]{GeckMalle}: If $F$ defines an $\mathbb{F}_q$-structure on $\G$, then $\G^{F^d}$ inherits an $\mathbb{F}_q$-structure but can also be considered as group over $\mathbb{F}_{q^d}$. Under this identification, the $d$-split Levi subgroups of $(\G,F)$ correspond to the $1$-split Levi subgroups of $(\G,F^d)$.
	
	We claim that similarly for a positive integer $k$, the $d$-split Levi subgroups of $(\G,F)$ correspond to the $d/(d,k)$-split Levi subgroups of $(\G,F^k)$. For this it suffices to observe that if $(\T,F)$ is a $d$-split torus, then $(\T,F^k)$ is a $d/(d,k)$-split torus. Indeed, $(\T,F^k)$ splits completely over $\mathbb{F}_{q^{\mathrm{lcm}(d,k)}}$ and if any non-trivial subtorus splits over $\mathbb{F}_{q^{ki}}$, then it also splits over $\mathbb{F}_{q^{(d,ki)}}$ which forces $d/(d,k) \mid i$. 
\end{remark}

Recall that $d$ is the order of $q$ modulo $\ell$. Additionally, we set $d_0:=d$ if $d$ is odd and $d_0:=d/2$ if $d$ is even.

\begin{lemma}\label{isolated}
	Assume that $s$ is quasi-isolated in $\G^\ast$ and $(\Levi,\lambda)$ a $d$-cuspidal pair of $G$ with $\lambda \in \mathcal{E}(L,s)$ as before. Then $\Z(\C_{L^\ast}(s))_\ell = \Z(L^\ast)_\ell$.
\end{lemma}

\begin{proof}
%	This can be done by considering the tables in Taylor.
%	The element $s$ is in a $d$-split maximal torus $\T^\ast$ of $\G^\ast$ for $d \in \{1,2\}$ and $L(s)$ is an $e$-split Levi subgroup of $G(s)$.
By the properties of $d$-cuspidal pairs, see \cite[Remark 2.2]{Cabanesgroup}, we have $\Z^\circ(\C_{\Levi^\ast}^\circ(s))_{\Phi_d}=\Z^\circ(\Levi^\ast)_{\Phi_d}$. For a torus $\T$ defined over $\mathbb{F}_q$ and $E$ a set of integers we let $\T_{\phi_E}$ the unique maximal $\Phi_E$-subgroup, see \cite[Definition 13.3, Proposition 13.5]{MarcBook}.
 To prove the claim it is therefore by \cite[Proposition 13.12]{MarcBook} sufficient to show that $\Z^\circ(\C_{\Levi^\ast}^\circ(s))_{\Phi_d}=\Z^\circ(\C_{\Levi^\ast}^\circ(s))_{\Phi_E}$, where $E:=E_{q,\ell}$.

Observe that $\Z^\circ(\C_{\Levi^\ast}^\circ(s))_{\Phi_d}$ is contained in a Sylow $d$-torus $\mathbf{S}$ of $\C_{\G^\ast}^\circ(s)$. Hence, it enough to show the statement with $\Z^\circ(\C_{\Levi^\ast}^\circ(s))$ replaced by $\Z^\circ(\C_{\C_{\G^\ast}^\circ(s)}(\mathbf{S}))$. For $\mathbf{H}$ simple of classical type the connected center of the minimal $d$-split Levi subgroups are (up to multiplication or division by a factor $q \pm 1$) of the form $(q^{d_0}-\varepsilon_d)^{a(d)}$, where $d_0 \in \{2d,d,d/2\}$, $a(d) \in \mathbb{N}$ and $\varepsilon_d \in \{\pm 1\}$, see \cite[Example 3.5.15]{GeckMalle}. Since $\ell \neq 2$, the $E_{q,\ell}$-part of the center is equal to its $d$-part by Lemma \ref{cyclotomic}. We can write the adjoint quotient $(\C_{\G^\ast}^\circ(s))_{\mathrm{ad}}$ of $\C_{\G^\ast}^\circ(s)$ as $(\C_{\G^\ast}^\circ(s))_{\mathrm{ad}}=\mathbf{H}_1 \times \mathbf{H}_2 \times \mathbf{H}_3$ for simple (or trivial) algebraic groups $\mathbf{H}_1, \mathbf{H}_2,\mathbf{H}_3$ with $(\C_{G^\ast}^\circ(s))_{\mathrm{ad}}=H_1 \times H_2 \times H_3$ or $(\C_{G^\ast}^\circ(s))_{\mathrm{ad}} \cong H_1^{F^2} \times H_3$, see \cite[Table 2]{MalleCuspidal}. In particular, by Remark \ref{scalar descent} the polynomial order of $\mathbf{S}_{\mathrm{ad}}$, the image of $\mathbf{S}$ in the adjoint quotient $(\C_{\G^\ast}^\circ(s))_{\mathrm{ad}}$, is a product of $\Phi_e$-polynomials with $e \mid 4 d$. Since $2 \neq \ell$ it follows that $\ell \nmid \Phi_{4d}(q)$ by Lemma \ref{cyclotomic}. Since $\Phi_1,\Phi_2$ are the only cyclotomic polynomials that possibly divide the generic order of $\Z^\circ(\C_{G^\ast}^\circ(s))$ by \cite[Table 2]{MalleCuspidal} and $2 \neq \ell$ the claim follows from \cite[Proposition 13.7]{MarcBook}.
%Since $\Z(L(s))_{\Phi_{e}}=\Z(L)_{\Phi_e}$ it follows that  $\Z(L(s))_{\ell}=\Z(L)_{\ell}$.
%By Lemma \ref{bijection} for the Levi subgroup $s \in \Levi$ we have 
%$$\Levi=\C_{\G}(\Z^\circ(\C_{\Levi}(s))).$$
%In particular, $\Z^\circ(\C_{\Levi}(s)) \subset \Z^\circ(L)$. Moreover, $s$ is isolated in $\G$ if and only if $\Z^\circ(\C_{\G^\ast}(s))=\Z^\circ(\G^\ast)$. Since $\C_{\Levi^\ast}(s)$ is a Levi subgroup of $\C_{G^\ast}(s)$ it follows that the component group of the centre of $C_{\Levi^\ast}(s)$ is a subgroup of the component of the centre of $\C_{G^\ast}(s)$. Since $s$ is isolated in $\G^\ast$, it follows that $\C_{\G^\ast}(s) / \Z^\circ(\G)$ is a semisimple group of classical type. In particular, since $\ell \geq5$ the component group of the centre of $\C_{G^\ast}(s)$ is an $\ell'$-group. Therefore, $\Z(\C_{\Levi}(s))^F_\ell=\Z^\circ(\C_{\Levi}(s))^F_\ell$ and so $\Z(\C_{\Levi}(s))_\ell \subset Z(L)_\ell$. Any element in $\Z(\Levi^\ast)$ will on the other hand centralize $s$ and so $\Z(\Levi^\ast) \subset \Z(\C_{\Levi^\ast}(s))$.
\end{proof}

%Lemma \ref{isolated} remains true for quasi-isolated elements in type $D$ by Malle, Extension of cuspidal characters, Table 2. This is because $[G^\circ(s),G^\circ(s)]$ is as described in Lemma \ref{isolated} and $\Z^\circ(G^\circ(s))$ is a product of $\Phi_1$, $\Phi_2$-polynomials.

\begin{corollary}\label{cabanes}
	Let $b=b_{\G^F}(\Levi,\lambda)$ be a quasi-isolated block associated to the $d$-cuspidal pair $(\Levi,\lambda)$. Then the subgroup $Q:=\Z(L)_\ell$ is a Cabanes subgroup of a defect group of $b$.
\end{corollary}

\begin{proof}
Since $\ell \nmid |\tilde{G}/G \Z(\tilde{G})|$, it suffices to prove the statement for a $d$-cuspidal pair $(\tilde{\Levi},\tilde{\lambda})$ of $(\tilde{\G},F)$ covering $(\Levi,\lambda)$. The defect group of the block of $\tilde{G}$ associated to this $d$-cuspidal pair is isomorphic to the defect group of the block of the associated unipotent $d$-cuspidal pair $(\tilde{L}(\tilde{s}),\tilde{\lambda}(\tilde{s}))$ of $\tilde{G}(\tilde{s})$, see \cite[Proposition 5.1]{Marc}. We know that $Z:=\Z(\tilde{L}(\tilde{s}))_\ell$ is a normal abelian subgroup of a defect group $D$ of the associated unipotent block, see \cite[Lemma 4.5]{CabEng}. Moreover, by the proof of \cite[Lemma 4.5]{CabEng}, we have $D \cap \C_{G}(Z) \leq Z$. This implies that $Z$ is a maximal abelian normal subgroup of $D$ and thus the Cabanes subgroup of $D$. We also observe that $(\Z(\tilde{L})_\ell,b_{\tilde{L}}(\tilde{\lambda}))$ is a $b_{\tilde{G}}(\tilde{\Levi},\tilde{\lambda})$-subpair by \cite[Theorem 2.5]{Marc}. Since $Z$ and $\Z(\tilde{L})_\ell$ have the same order by Lemma \ref{isolated}, this shows the claim.
\end{proof}

Note that $\Levi=\C_{\G}(Q)$ by \cite[Proposition 13.19]{MarcBook}.

\begin{lemma}\label{local}
In the situation of Corollary \ref{cabanes}, the following hold.
	\begin{enumerate}[label=(\alph*)]
		\item The block $b_{\Levi^F}(\Levi,\lambda)$ is of central defect and there exists a bijection $$\Z(\Levi^\ast)_\ell^F=\Z(\C_{\Levi^\ast}(s))_\ell^F \to \Irr(b_{\Levi^F}(\Levi,\lambda)), {t} \mapsto \hat{t} \otimes \lambda.$$
		\item  Any height zero character $\chi \in \Irr_0(G,b)$ is a constituent of $R_\Levi^\G(\lambda \hat{t})$ with $t\in \Z(\Levi^\ast)_\ell^F$.
	\end{enumerate}
	
\end{lemma}

\begin{proof}
	By Lemma \ref{isolated}, $(Q,b_L(\lambda))$ is a self-centralizing $b$-subpair. In particular, $b_L(\lambda)$ is of central defect. From this part (a) follows from the remarks after \cite[Definition 1.13]{exceptional} and \cite[(8.19)]{MarcBook}. Part (b) follows from Lemma \ref{ht zero} and the compatibility of Lusztig induction with regular embeddings. 
%	Again, according to Enguehard there exists a height preserving bijection between the block $c$ of $\Levi^F$ and a corresponding unipotent block $c(s)$ of $\Levi(s)^F$. By Lemma \ref{lemma1} we obtain
%	$$\Irr_0(L(s),c(s)) \subset \cup_{t \in Z(L(s)^\ast)_\ell} \{ \lambda_s \otimes \hat{t} \},$$
%	in particular since all characters on the right hand side have the same character degree this inclusion must be an equality. We obtain a bijection $$\Z(\C_{L^\ast}(s))_\ell \to \Irr_0(L(s),c(s)).$$
%	Moreover, by Theorem \ref{MalleNavarro} the block $c(s)$ has central defect group $\Z(L(s))_\ell^F$.  It follows by the properties of Enguehard's bijection that the defect group of the block $c$ has the same order as the block $c(s)$. Now, Malle-Navarro Thm. 6.1 together with the structure of Levi subgroups yields that the block $b_L(L,\lambda)$ of $L$ is nilpotent. So $|Z(L^\ast)|_\ell^F \leq |Z(L(s)^\ast)|_\ell^F$ but equality does not necessarily hold. Assuming $s$ isolated it does by Remark \ref{isolated}.
\end{proof}

\subsection{Properties of blocks associated to height zero characters}

The aim of this section is to give an alternative proof of Lemma \ref{local}(b) using slightly different techniques.

\begin{lemma}\label{height0}
	Let $\G$ and $F$ as in Lemma \ref{dHC}. Suppose that $b$ is an $\ell$-block of $G$ associated to the semisimple element $s \in G^\ast$. Then 
	$\mathcal{E}(G,s) \cap \Irr_0(G,b) \neq \emptyset$.
\end{lemma}

\begin{proof}
	The block $b$ is parametrized by the $G$-conjugacy class of a $d$-cuspidal pair $(\Levi,\lambda)$. The defect group of $b$ has order $d(b):=|W_{\C_{G^\ast}(s)}(\C_{\Levi^\ast}(s))|_\ell |\Z(\C_{L^\ast}(s))|_\ell$ by \cite[Lemma 4.16]{Marc}. The characters in 	$\mathcal{E}(G,s) \cap \Irr(G,b)$ are precisely the constitutents of the $d$-Harish-Chandra series of $(\Levi,\lambda)$, see \cite[Theorem 4.1]{Marc}. Via Jordan decomposition the $d$-Harish-Chandra series of $(\Levi,\lambda)$ is mapped to the unipotent $d$-Harish-Chandra series of $(\C_{\Levi^\ast}(s),\lambda_s)$, see \cite[Theorem 4.7.2]{GeckMalle}. The degree formula following \cite[Theorem 4.6.24]{GeckMalle} shows that the character $\psi \in \Irr(\C_{G^\ast}(s),1)$ corresponding to the trivial character of the relative Weyl group satisfies $\psi(1)_\ell=\lambda_s(1)_\ell |\C_{G^\ast}(s):\C_{L^\ast}(s)|_\ell |W_{\C_{G^\ast}(s)}(\C_{L^\ast}(s),\lambda_s)|^{-1}_\ell$ and $\lambda_s$ has $\ell$-central defect (remarks after \cite[Corollary 4.6.16]{GeckMalle}), i.e. $\lambda_s(1)_\ell=|\C_{L^\ast}(s):\Z(\C_{L^\ast}(s))|_\ell$. Its Jordan correspondent therefore has degree, $$\chi(1)_\ell=|G:\C_{G^\ast}(s)|_\ell |\C_{G^\ast}(s):\Z(\C_{L^\ast}(s))|_\ell |W_{G(s)}(\C_{L^\ast}(s),\lambda_s)|^{-1}_\ell=|G| d(b)^{-1}_\ell.$$ Hence, $\chi \in \Irr(G,b) \cap \mathcal{E}(G,s)$ is of height zero.
\end{proof}

\begin{proposition}
		Let $\G$ and $F$ as in Lemma \ref{height0}. Let $b=b_G(\Levi,\lambda)$ be an $\ell$-block of $G$ which covers a quasi-isolated $\ell$-block of $[\G,\G]^F$. Then we have $\Irr_0(G,b) \subset \cup_{t \in \Z(\Levi^\ast)_\ell} \mathcal{E}(G,st)$. Moreover, any $\chi \in \Irr_0(G,b) \cap \mathcal{E}(G,st)$ is contained in the $d$-Harish-Chandra series of $(\Levi,\hat{t} \lambda)$.
\end{proposition}	

\begin{proof}
	Assume that $\Irr_0(G,b) \cap \mathcal{E}(G,st) \neq \emptyset$ and let $\mathbf{K}:=\G(t)$. By \cite[Theorem 9.16]{MarcBook} the map $ \pm R_{\mathbf{K}}^{\G} \hat{t}:\mathcal{E}(K,s) \to \mathcal{E}(G,st)$ is a bijection.
	According to \cite[Theorem 2.5]{Marc}, there exists a unique block $c$ of $K$ such that $\pm R_{\mathbf{K}}^{\G} \hat{t}$ maps $\Irr(K,c)$ to $\Irr(G,b)$. In \cite{EnguehardJordandecomp}, the author characterizes the block $c$ by a certain $d$-cuspidal pair which he obtains from (the Jordan correspondent) of the $d$-cuspidal pair $(\Levi,\lambda)$, see also the proof of Lemma \ref{ht zero}. In our situation we make instead use of the fact that $\mathcal{E}(G,s)$ contains a height zero character in order to describe the block $c$.
	
	By Lemma \ref{height0}, it follows that the map $\mathcal{E}(K,s) \to \mathcal{E}(G,st)$ from above is height preserving. In particular, the defect groups of $b$ and $c$ have the same order. Furthermore, by \cite[Theorem 2.5]{Marc} and \cite[Proposition 13.19]{MarcBook} since $\mathbf{K}$ is $E_{q,l}$-split we obtain that $(\Z(K)_\ell,c)$ is a $b$-Brauer pair. Let $(D,c_D)$ be a maximal $c$-Brauer pair so that $D \leq K$. Hence, $(1,b) \lhd (\Z(K)_\ell,c) \lhd (D,b_D)$. By transitivity of Brauer pairs, $(D,c_D)$ is also a $b$-Brauer and since the defect groups of $c$ and $b$ have the same order it follows that $(D,c_D)$ is a maximal $b$-Brauer pair. In particular, $\Z(K)_\ell$ is a normal abelian subgroup of $D \leq M$. Since $\Z(L)_\ell=\Z(\C_{L^\ast}(s))_\ell$ by Lemma 3.3 is the unique maximal normal abelian subgroup of a defect group of $b$ we may by replacing $(\Levi,\lambda)$ by a $G$-conjugate assume that $(\Z(K)_\ell,c) \leq (\Z(L)_\ell,b_L(\lambda))$. This implies $\Levi=\C_\G(\Z(L)_\ell) \subset \C_\G(\Z(K)_\ell)=\mathbf{K}$ since both $\Levi$ and $\mathbf{K}$ are $E_{q,l}$-split (\cite[Proposition 13.19]{MarcBook}). Dually this means $\Levi^\ast \subset \mathbf{K}^\ast$ which implies $\Z(\mathbf{K}^\ast)\leq \Z(\Levi^\ast)$ and therefore $t \in \Z(\mathbf{\Levi}^\ast)$.
	
	Since $\Levi$ is an $e$-split Levi subgroup of $\mathbf{K}$ and $(\Z(L)_\ell, b_L(\lambda))$ is a $c$-Brauer pair it follows from \cite[Theorem 4.1]{Marc} that $c=R_\Levi^{\mathbf{K}}(b_L(\lambda))$ in the notation of \cite[Theorem 4.1]{Marc}.
	% if $c'=R_\Levi^{\mathbf{M}}(b_L(\lambda))$ then $(\Z(L)_\ell, b_L(\lambda))$ is a $c'$-Brauer. But the sets of $c$-Brauer pairs and $c'$-Brauer pairs are disjoint unless $c=c'$.
	
	Hence, every character of $\mathcal{E}(K,s)$ appears as a constituent of $R_\Levi^{\mathbf{K}}(\lambda)$, see \cite[Theorem 4.1]{Marc}. By transitivity of Lusztig induction, every character of $\mathcal{E}(G,st)$ therefore appears as a constituent of $R_{\mathbf{K}}^{\G} \hat{t} R_\Levi^{\mathbf{K}}(\lambda)=R_\Levi^{\G}(\hat{t} \lambda)$.
	%	In addition, by CE99, Proposition 5.1, the defect group of $b$ (resp. $c$) are isomorphic to the corresponding ones of $G(s)$ (resp. $G(st)$). 
\end{proof}

%\begin{corollary}
%	Assume that $s$ is isolated in $\G^\ast$.
%\end{corollary}

\section{Constructing an AM-bijection}\label{discentre}

\subsection{Conditions on the local block}
%We describe the Clifford theory for blocks. Recall that Jordan $d$-cuspidal and $d$-cuspidal coincide for $\ell \geq 5$ by \cite[Remark 2.2]{Cabanesgroup}. Moreover, the set of blocks of $\G^F$ are in bijection with the $\G^F$-conjugacy classes of $d$-cuspidal pairs, see \cite[Theorem A]{Cabanesgroup}. We let $\iota: \G \hookrightarrow \Gtilde$ a regular embedding and $(\tilde{\Levi},\tilde{\lambda})$ a $d$-cuspidal pair as before.

%While the property of being a nilpotent block may not be preserved under passage to normal subgroups, the property of being of central defect is.
As before, we let $(\Levi,\lambda)$ be a $d$-cuspidal pair of a simple, simply connected algebraic group $\G$ of classical type $B_n$, $C_n$ or $D_n$. Denote by $\mathcal{B}$ the subgroup of $\mathrm{Aut}_{\mathbb{F}}(\G^F)$ generated by field and graph automorphisms as in \cite[Section 2.A]{TypeB}. For simplicity, we denote $N:=\N_G(\Levi)$, $\hat{N}:=\N_{G\mathcal{B}}(\Levi)$ and $\tilde{N}:=N \tilde{L}$.

\begin{definition}\label{def star}
	We say that a character $\chi \in \Irr(G)$ (resp. $\psi \in \Irr(N))$ satisfies $A'(\infty)$, if $(\tilde{G}\mathcal{B})_\chi=\tilde{G}_\chi \mathcal{B}_\chi$ (resp. $(\tilde{N}\hat{N})_\psi=\tilde{N}_\psi \hat{N}_\psi$).
\end{definition}

\begin{proposition}[Enguehard]\label{enguehard}

Let $(\Ltilde,\tilde{\lambda})$ be a $d$-cuspidal pair of $(\Gtilde,F)$ covering $(\Levi,\lambda)$.
\begin{enumerate}[label=(\alph*)]
	\item There exists $\tilde{\chi} \in \Irr(\tilde{G})$ with $\langle R_{\tilde{\Levi}}^{\Gtilde}(\tilde{\lambda}),\tilde{\chi} \rangle = \pm 1$ and the degree of $\tilde{\chi}$ is different from the degrees of all other irreducible constituents of $R_{\tilde{\Levi}}^{\tilde{\G}}(\tilde{\lambda})$.
	\item We have $\N_G(\Levi,\Res^{\tilde{L}}_L(\tilde{\lambda}))=\N_G(\Levi,\lambda)$. 
\end{enumerate}	
\end{proposition}

\begin{proof}
	Part (a) is \cite[2.3.1]{EnguehardJordandecomp} while part (b) is \cite[Proposition 2.3.2]{EnguehardJordandecomp}.
\end{proof}

\begin{corollary}\label{coro}
Let $(\Levi,\lambda)$ be as in the previous proposition. Then we have:
	\begin{enumerate}[label=(\alph*)]
		\item $(N\tilde{L})_\lambda=N_\lambda \tilde{L}_\lambda$.
		\item There exists some $\tilde{L}$-conjugate $\lambda_0$ of $\lambda$ such that
		$$(\hat{N} \tilde{L})_{\lambda_0}=\hat{N}_{\lambda_0} \tilde{L}_{\lambda_0}.$$
%		$$(\hat{N}\tilde{N})_{\Ind_{L}^N(\lambda_0)}=\hat{N}_{\Ind_{L}^N(\lambda_0)} \tilde{N}_{\Ind_{L}^N(\lambda_0)}.$$
%		In particular, 
		\item The stabilizer of $b_G(\Levi,\lambda)$ and of $\Ind_L^N(\lambda)$ in $\N_{\tilde{G} \mathcal{B}}(\Levi)= \hat{N} \tilde{L}$ is the same. In particular, the number of $\tilde{G}$-conjugate blocks to $b_G(\Levi,\lambda)$ is $|\tilde{L}:\tilde{L}_\lambda|$.
	\end{enumerate}
\end{corollary}

\begin{proof}
Part (a) is a reformulation of Proposition \ref{enguehard}(b). For part (b), let $\tilde{\chi} \in \Irr(\tilde{G})$ as in Proposition \ref{enguehard}(a).
%As in the proof of \cite[Prop. 2.3.2]{EnguehardJordandecomp} let $\chi \in \Irr(G \mid \tilde{\chi})$ with $\langle R_L^G(\lambda),\chi \rangle = \pm 1$.
By the introduction of \cite{TypD} there exists a character $\chi \in \Irr(G \mid \tilde{\chi})$ which satisfies $A'(\infty)$. In particular, there exists some $\tilde{L}$-conjugate $\lambda_0$ of $\lambda$ such that $\langle R_\Levi^\G(\lambda_0),\chi \rangle = \pm 1$, see the proof of \cite[Prop. 2.3.2]{EnguehardJordandecomp}.

 Now assume that ${}^{\tilde{l} \hat{n}} \lambda_0=\lambda_0$ for $\tilde{l} \in \tilde{L}$ and $\hat{n} \in \hat{N}$. Then ${}^{\hat{n}} \tilde{\lambda}=\tilde{\lambda} \hat{z}$ for some $z \in \Z(\tilde{G}^\ast)$. From this it follows that $\tilde{\chi}^{\hat{n}} \hat{z}^{-1}$ and $\tilde{\chi}$ are both constituents of the same degree of $R_{\tilde{\Levi}}^{\tilde{\G}}(\tilde{\lambda})$. Hence, $\tilde{\chi}^{\hat{n}} \hat{z}^{-1}=\tilde{\chi}$ and so $\chi^{\hat{n}}=\chi$ by the $A'(\infty)$ condition. This means that the $d$-Harish-Chandra series $(\Levi,\lambda)$ and $(\Levi,{}^{\hat{n}}\lambda)$ have a non-trivial intersection hence these pairs must be $N$-conjugate by \cite[Theorem A]{Cabanesgroup}. In particular, $\tilde{l}$ stabilizes the $N$-orbit of $\lambda_0$ and thus $\tilde{l} \in \tilde{L}_{\lambda_0}$ by part (a).
% This shows
%	$$(\hat{N}\tilde{N})_{\Ind_{L}^N(\lambda_0)}=\hat{N}_{\Ind_{L}^N(\lambda_0)} \tilde{N}_{\Ind_{L}^N(\lambda_0)}.$$
%	Now the first part shows the required equality.	
Part (c) follows similarly from \cite[Theorem A]{Cabanesgroup}.
\end{proof}

%Note that the character $\lambda_0$ from Corollary \ref{coro} satisfies 
%By Proposition \ref{enguehard} and Corollary \ref{coro},

%We will later give an independent proof (if $\G$ is of type $B$ or $C$) that $(\Levi,\lambda)$ satisfies the conclusion of Corollary \ref{coro}(a)+(b). 

%Note that $\hat{W}(\lambda)=\hat{W}(\lambda_0)$ if $\lambda_0$ is a character which has $A'(\infty)$ and lies in the same $\tilde{L}$-orbit as $\lambda$.
We expect that in many cases the properties of the $d$-cuspidal pair $(\Levi,\lambda)$ (or more precisely the characters of $[\Levi,\Levi]^F$ lying below $\lambda$) determine the properties of the block of $N$ covering it. In view of the similarity to Definition \ref{def star}, we make the following definition:

\begin{definition}\label{cuspstar}
We say that a $d$-cuspidal pair $(\Levi,\lambda)$ of $G$ satisfies $A'(\infty)$, if $(\hat{N} \tilde{L})_\lambda=\hat{N}_\lambda \tilde{L}_\lambda$.
\end{definition}

Define $\hat{W}:=\N_{G \mathcal{B}}(\Levi)/L=\hat{N}/L$ and $W(\theta):=\N_{G}(\Levi,\theta)/L$ for $\theta \in \Irr(H)$ with $L \leq H \leq \tilde{L}$. The definition of extension maps used in the following can be found for instance in \cite[Section 2]{TypeB}. We will work with the following assumption.

\begin{assumption}\label{assumption}
	We suppose that $(\Levi,\lambda)$ is a $d$-cuspidal pair of $G$ which satisfies $A'(\infty)$. Moreover we assume the following:
	\begin{enumerate}[label=(\roman*)]
		\item 

	There exists an $\hat{N}_\lambda$-equivariant extension map $\Lambda$ from $L$ to $N$ for $\Irr(L,b_L(\lambda))$.
	\item 
	For $\theta \in \Irr(L,b_L(\lambda))$, $\widetilde{\theta} \in \Irr(\tilde{L}_{\theta} \mid \theta)$ and $\eta_0 \in \Irr(W(\widetilde{\theta}))$ there exists some $\hat{W}(\theta)_{\eta_0}$-stable $\eta \in \Irr(W(\theta)\mid \eta_0)$ with $\langle \Res_{W(\widetilde{\theta})}^{W(\theta)}(\eta), \eta_0 \rangle=1$. 
%	appears with multiplicity one in $\eta_0$.
	% suffices to have (b) only for lambda?
		\end{enumerate}
\end{assumption}

 \begin{remark}\label{rem}
\begin{enumerate}[label=(\alph*)]
	\item Note that condition (ii) is always satisfied when $W(\theta)=W(\tilde{\theta})$.
	\item The multiplicity one condition in (ii) is always satisfied whenever $\G$ is not of type $D_{2n}$ (proof of \cite[Theorem 4.2]{TypeB})
	 or when $d=1$ by \cite[Corollary 13.13]{Bonnafe2}.
	\item Assumption \ref{assumption} was checked in \cite{TypD} when  $d=1$ with $G=D_{n}(q)$.
\end{enumerate}
\end{remark}
\subsection{Constructing an AM-bijection}

Recall from \cite[Proposition 13.16,13.19]{Enguehard} that $\Levi=\C_{\G}(\Z(\Levi)_\ell^F)$. In particular, there exists by \cite[Theorem 9.19,9.20]{NavarroBook} a unique block $B$ of $\N_G(\Levi)$ covering $b_{L}(\lambda)$. The following proposition can be seen as a blockwise version of \cite[Proposition 1.12]{TypD}.

\begin{proposition}\label{star}
%	Assume that the Assumptions of Prop. 1.12 Späth Type D are satisfied for the $d$-cuspidal character $\lambda$.
Assume that the block $b_G(\Levi,\lambda)$ is a quasi-isolated block of $G$. Moreover, suppose that the block $b_{L}(\lambda)$
%	 is a block of central defect (automatic if $b_{\tilde{L}}(\tilde{\lambda})$ already has central defect)
%	$\ell \nmid \Z(L) / \Z(L)^\circ|$, $\ell \geq 5$
satisfies Assumption \ref{assumption}. Then every character of $B$ has a $\tilde{L}_b$-conjugate which satisfies $A'(\infty)$. 
\end{proposition}

\begin{proof}
By Lemma \ref{local} we have $\Irr(L,b_L(\lambda))=\{ \lambda \otimes \hat{t} \mid t \in \Z(L^\ast)_\ell \}$. The character $\lambda$ is the canonical character of the block $b_L(\lambda)$, see \cite[Theorem 9.12]{NavarroBook}. Together with Assumption \ref{assumption} this implies $$(\hat{N} \tilde{L})_{\theta}\leq (\hat{N} \tilde{L})_{\lambda} = \hat{N}_{\lambda} \tilde{L}_{\lambda}$$ for $\theta:=\lambda \otimes \hat{t}$. Since $\tilde{L}_\lambda=\tilde{L}_{\theta}$, we deduce that $(\hat{N} \tilde{L})_{\theta}=\hat{N}_{\theta} \tilde{L}_{\theta}$.
The extension map $\Lambda$ yields by Clifford theory a parametrization $$\Irr(W(\theta)) \to \Irr(N \mid \theta), \eta \mapsto \Ind^{N}_{N_{\theta}}(\eta \Lambda(\theta)).$$
For such a character $\psi:=\Ind^{N}_{N_{\theta}}(\eta \Lambda(\theta))$ let $\eta_0$ be a constituent of $\Res_{W(\tilde{\theta})}^{W(\theta)}(\eta)$. For $t \in \tilde{L}_\theta$, there exists a unique character $\nu_t \in \Irr(W(\theta)/W(\tilde{\theta}))$ such that ${}^t \Lambda(\theta)=\Lambda(\theta) \nu_t$. It follows that the map $$\tilde{L}_{\theta}/\tilde{L}_{\Lambda(\theta)} \to \Irr(W(\theta)/W(\tilde{\theta})), t \mapsto \nu_t,$$
is bijective, see the proof of \cite[Theorem 4.3]{TypeC}. Hence, by replacing $\psi$ by a $\tilde{L}_{\theta}=\tilde{L}_\lambda$-conjugate, we can assume that $\eta$ satisfies Assumption \ref{assumption}(ii). Assume now that some $\tilde{l} \hat{n}^{-1}$ with $\tilde{l} \in \tilde{L}$ and $\hat{n} \in \hat{N}$ stabilizes $\psi$. In particular, $\tilde{l} \hat{n}^{-1}$ stabilizes the $N$-orbit of $\theta$. We deduce that $\tilde{l} \in \tilde{L}_{\theta}$ and there exists some $n \in N$ such that $\hat{n}':=\hat{n} n \in \hat{N}_{\theta}$. We deduce that
$$ \Ind_{N_{\theta}}^{N}(\eta^{\hat{n}'} \Lambda(\theta))=\psi^{\hat{n}'}=\psi^{\tilde{l}}=\Ind_{N_{\theta}}^{N}(\eta \nu \Lambda(\theta)),$$
for some linear character $\nu \in \Irr(W(\theta)/W(\tilde{\theta}))$.
From this, we deduce that $\eta^{\hat{n}'}=\eta \nu$ and as $\eta \in \Irr(W(\theta))$ is $W(\theta)\hat{W}(\theta)_{\eta_0}$-stable, we have $\nu=1$. In particular, $\psi \in \Irr(N,B)$ satisfies $A'(\infty)$.
%Now observe that $\tilde{\lambda} \otimes \hat{\tilde{t}}$ covers $\theta$. The group $W_{G}(\theta)/W_G(\tilde{\theta})$ is an $\ell'$-group by assumption and $W_{G}(\theta) \subset W_G(\lambda)$. Hence $W_{G}(\theta)/W_{G}(\tilde{\theta})$ is a subgroup of $W_G(\lambda)/W_G(\tilde{\lambda})$.
%A criterion for a character lying over a character of $L$ to have the star condition is e.g. in CS Type. For this it is important that the underlying character $\theta$ has the star-condition (modulo $N$). Note that $\theta$ lies in the same block as $\lambda$ and $\tilde{L}_{\theta} = \tilde{L}_\lambda$. Hence the character $\theta$ has $A'(\infty)$, whenever $\lambda$ has it.
%In particular, if Späth's critierion holds and $\lambda$ has Star, then there exists a character of $N$ covering it, which has Star as well.
%
%This is particular interesting, when $\Lambda(\lambda)$ and $\lambda$ have different inertia groups in $\tilde{L}$. In this case, Späth lifting implies that all extensions must have the star condition, or equivalenty, $[\tilde{G},\mathcal{B}_\chi] \subset \tilde{G}_\chi$. This is only possible if $\mathcal{B}_\chi$ acts trivially on $\tilde{G}/G_\chi$.
%
% Hence, let $\mu \in \Irr(B)$. Then $\mu$ lies over a character $\theta=\lambda \otimes \hat{z}$, which has the $A'(\infty)$-condition, since $\lambda$ has the $A'(\infty)$ condition. Hence, by Späth lifting, there exists $\tilde{l} \in \tilde{L}_{\theta}=\tilde{L}_{\lambda}=\tilde{L}_b$ such that $\mu^{\tilde{l}}$ has the star condition. In particular, $\mu^{\tilde{l}}$ lies in the block $B$. 
\end{proof}
Suppose that there exists an $\hat{N}_\lambda$-equivariant extension map $\Lambda$ from $L$ to $N$ for $\Irr(L,b_L(\lambda))$ as in Assumption \ref{assumption}. This then yields an $\hat{N}_\lambda$-equivariant extension map $\tilde{\Lambda}$ for $\tilde{L} \lhd \tilde{N}$ for $\Irr(\tilde{L},b_L(\lambda))$, the set of characters of $\tilde{L}$ lying over a character of $\Irr(L,b_L(\lambda))$, which is compatible with the action of linear characters of $\Irr(\tilde{N}/N)$, see the proof of \cite[Theorem 4.2]{TypeB}.

\begin{lemma}\label{max ext}
	Suppose Assumption \ref{assumption} and let $B$ be as before. Then every character of $\Irr(N,B)$ extends to its inertia group in $\tilde{N}$.
\end{lemma}

\begin{proof}
	Let $\psi \in \Irr(N)$ and $\theta \in \Irr(L \mid \psi)$. We can assume that $\theta$ is $\tilde{L}$-stable since otherwise $\tilde{N}_\psi/N \Z(\tilde{G})$ is cyclic and the claim trivially holds. We can write $\tilde{\psi} \in \Irr(\tilde{N} \mid \psi)$ as $\tilde{\psi}=\Ind_{\tilde{N}_{\theta} }^{\tilde{N}}(\tilde{\Lambda}(\tilde{\theta}) \tilde{\eta})$ for $\tilde{\theta} \in \Irr(\tilde{L} \mid \theta)$ and $\tilde{\eta} \in \Irr(W_{\tilde{G}}(\tilde{\Levi},\tilde{\theta}))$. By Mackey's formula,
	$$\Res_{N}^{\tilde{N}}(\tilde{\psi} )=\Ind_{N_{\tilde{\theta}}}^{N} \Res^{\tilde{N}_{\tilde{\theta}}}_{N_{\tilde{\theta}}}(\tilde{\eta} \tilde{\Lambda}(\tilde{\theta}) ).$$
	Now, $\Ind^{W(\lambda)}_{W({\tilde{\lambda}})}(\tilde{\eta})$ is multiplicity free as well by Assumption \ref{assumption}(ii) and Clifford theory.
	By the construction of the extension map $\tilde{\Lambda}$ in \cite[Theorem 4.2]{TypeB} we observe that $$\Ind^{N_{\theta}}_{N_{\tilde{\theta}}}( \Res^{\tilde{N}_{\tilde{\theta}}}_{N_{\tilde{\theta}}}(\tilde{\eta} \tilde{\Lambda}(\tilde{\theta}) ))=\sum_{\eta \in \Irr(W(\theta) \mid \tilde{\eta})} \eta\Lambda(\theta).$$
	By Clifford correspondence, $\Res_N^{\tilde{N}}(\tilde{\psi})$ is therefore multiplicity free. Clifford theory now gives the claim.
\end{proof}

Note that since $\ell \nmid |\tilde{G}/G \Z(\tilde{G})|$ every character of $\tilde{G}$ covering a character in $\Irr_0(G,b)$ is a height zero character of a block lying over $b$. We denote by $\Irr_0(\tilde{G},b)$ the set of these characters. Similarly, $\Irr_0(\tilde{N},B)$ denotes the set of characters of $\tilde{N}$ covering a character of $\Irr_0(N,B)$.

Recall the notion of AM-good relative to the Cabanes subgroup of its defect group from \cite[Remark 9.6]{Jordan3}.
Applying the criterion from \cite[Theorem 2.4]{Brough} yields the following:

\begin{theorem}\label{bijection}
	Let $\G$ be of type $B_n$ or $C_n$, $p$ an odd prime and $\ell \geq 5$. Moreover, let $b=b_G(\Levi,\lambda)$ be a quasi-isolated block and $B$ as above. Suppose that Assumption \ref{assumption} is satisfied.
%	nd that $\mathrm{Out}(G)_\mathcal{B}$ is abelian, where $\mathcal{B}$ is the $\tilde{G}$-orbit of $b$.
	Then there exists an automorphism equivariant bijection $f:\Irr_0(G,b) \to \Irr_0(N,B)$. In particular, the block $b$ is AM-good relative ot the Cabanes subgroup of its defect group.
\end{theorem}

\begin{proof}
%		We first construct an $\N_{\tilde{G}\mathcal{B}}(Q) \ltimes \Irr(\tilde{G}/G)$-equivariant bijection $\tilde{f}:\Irr_0(\tilde{G},\tilde{b}) \to \Irr_0(\tilde{N},\tilde{B})$.
		By Clifford theory, we have for $\tilde{\theta} \in \Irr(\tilde{L}, b_{L}(\lambda))$ a bijection
		$$\Irr(W_{\tilde{G}}(\tilde{\Levi},\tilde{\theta})) \to \Irr(\tilde{N} \mid \tilde{\theta}), \, \tilde{\eta} \mapsto \Ind_{\tilde{N}_{\tilde{\theta}}}^{\tilde{N}}(  \tilde{\eta} \Lambda(\tilde{\theta})),$$
		where $\tilde{\Lambda}$ is the extension map defined before Lemma \ref{max ext}.
		We compose this with the bijection
		$$\Irr(W_{\tilde{G}}(\tilde{\Levi}, \tilde{\theta}) ) \to \mathcal{E}(\G^F,(\tilde{\Levi},\tilde{\theta})),  \, \tilde{\eta} \mapsto R_{\tilde{\Levi}}^{\tilde{\G}}(\tilde{\theta})_{\tilde{\eta}}$$
		 from Proposition \ref{dHC}.
		Note that by the properties of both bijections, the tuples $(\tilde{\theta},\tilde{\eta})$ and $(\tilde{\theta}',\tilde{\eta}')$ give in both cases the same character if and only if they are $N_\lambda$-conjugate.
		 Hence, the composition yields an $(\N_{\tilde{G}\mathcal{B}}(Q) \ltimes \Irr(\tilde{G}/G))_b$-equivariant bijection $\tilde{f}:\Irr_0(\tilde{G},b) \to \Irr_0(\tilde{N},B)$. According to Proposition \ref{star} every character in $B$ satisfies $A'(\infty)$. Moreover by the introduction of \cite{TypD} every character of $b$ satisfies $A'(\infty)$. The claim follows therefore from the proof of \cite[Theorem 2.4]{Brough}.
\end{proof}

\section{Isolated blocks for groups of type $B$}\label{TypeB}

%Throughout this section we assume that $\G$ is a simple, simply connected group of type $B_n$.
%
%%\subsection{Some notation}
%
%
%
%\subsection{Construction of Levi subgroups}
Assume now that $\G$ is simple, simply connected of type $B$ defined over a field of odd characteristic $p$. By the reduction in \cite{Jordan3} it is enough for our purposes to consider isolated blocks. 

%\begin{remark}\label{scalar descent}
%		Recall \cite[Example 3.5.20]{GeckMalle}: If $F$ defines an $\mathbb{F}_q$-structure on $\G$, then $\G^{F^d}$ inherts an $\mathbb{F}_q$-structure but can also be considered as group over $\mathbb{F}_{q^d}$. Under this identification, the $d$-split Levi subgroups of $(\G,F)$ correspond to the $1$-split Levi subgroups of $(\G,F^d)$.
%		
%		We claim that similarly, the $d$-split Levi subgroups of $(\G,F)$ correspond to the $d_0$-split Levi subgroups of $(\G,F^2)$. For this it suffices to observe that if $(\T,F)$ is a $d$-split torus, then $(\T,F^2)$ is a $d_0$-split torus. Indeed, $(\T,F^2)$ splits completely $\mathbb{F}_{q^{2 d_0}}$ and if any non-trivial subtorus splits over $\mathbb{F}_{q^{2i}}$, then it also splits over $\mathbb{F}_{q^{(d,2i)}}$ which forces $d_0 \leq i$. 
%\end{remark}

\begin{lemma}\label{structure}
For a given isolated element $s \in G^\ast$ and $(\Levi,\lambda)$ a $d$-cuspidal pair of $\G$ with $\lambda \in \mathcal{E}(L,s)$ the structure of the Levi subgroup is one of the following:
\end{lemma}
% From the Tables in Taylor, we deduce that the following situation need consideration:

%\begin{table}[htbp]\label{table}
%	\caption{Quasi-isolated $\ell$-blocks in $B_n(q)$}   \label{tab:quasi-E6}
%	\[\begin{array}{|r|r|l|ll|}
%		\hline
%		\text{No.}& C_{\bG^*}(s)^F& d& L& C_{\bL^*}(s)^F\\
%		\hline\hline
%		1&        C_{k}(q) C_{n-k}(q)& d \text{ odd } & C_m(q) (q^d-1)^{n-m/d}& C_l(q) C_{m-l}(q) (q^d-1)^{n-m/d}\\
%		\hline
%		2&        C_{n/2}(q^2)& d \text{ odd }& C_m(q) A_1(q^d)^{n-m/2d}(q^{d}-1)^{n-m/2d}& C_{m/2}(q^2) (q^{2d}-1)^{n-m/2d}\\
%		\hline
%		3&   C_{k}(q) C_{n-k}(q)& d \text{ even } & C_{k}(q) C_{n-k}(q)&C_l(q) C_{m-l}(q) (q^{d_0}+1)^{n-m/d_0}\\
%		\hline
%		4&   C_{n/2}(q^2)& d \text{ even } d_0 \text{ odd }& C_m(q) A_1(q^{d_0})^{n-m/2d_0}(q^{d_0}+1)^{n-m/2d_0}&C_{m/2}(q^2) (q^{d}-1)^{n-m/2d_0}\\
%		\hline
%		5& C_{n/2}(q^2)& d_0 \text{ even } & C_{m}(q) (q^{d_0}+1)^{n-m/2d_0}& C_{m/2}(q^2) (q^{d_0}+1)^{n-m/2d_0}\\
%		\hline
%	\end{array}\]
%\end{table}

\begin{table}[htbp]
	\caption{Isolated $\ell$-blocks in $B_n(q)$}   \label{table}
	\[\begin{array}{|r|r|l|ll|}
		\hline
		\text{No.}& \C^\circ_{\bG^*}(s)^F& d_0 & \bL^F& \C^\circ_{\bL^*}(s)^F\\
		\hline\hline
		1 &        C_{k}(q) C_{n-k}(q)& & B_m(q) (q^{d_0}+\varepsilon)^{a}& C_l(q) C_{m-l}(q) (q^{d_0}+\varepsilon)^{a}\\
		\hline
		2&        C_{n/2}(q^2)& d_0 \text{ odd }& B_m(q) A_1(q^{d_0})^{a/2}(q^{d_0}+\varepsilon )^{a/2}& C_{m/2}(q^2) (q^{2d_0}-1)^{a/2}\\
		\hline
		3 & C_{n/2}(q^2)& d_0 \text{ even } & B_{m}(q) (q^{d_0}+1)^{a}& C_{m/2}(q^2) (q^{d_0}+1)^{a}\\
		\hline
	\end{array}\]
Here $l,m$ are integers, $a:=(n-m)/d_0$ and $\varepsilon:=(-1)^{d}$.
\end{table}

%For odd $d$.
%\begin{itemize}
%	\item $C_{G^\ast}(s)=C_{k}(q) C_{n-k}(q)$, $L^\ast=C_m(q) (q^d-1)^{n-m/d}$, $C_{L^\ast}(s)=C_l(q) C_{m-l}(q) (q^d-1)^{n-m/d}$
%	\item $C_{\G^\ast}(s)=C_{n/2}(q^2)$, $L^\ast=C_m(q) A_1(q^d)^{n-m/2d}(q^{d}-1)^{n-m/2d}$, $C_{L^\ast}(s)=C_{m/2}(q^2) (q^{2d}-1)^{n-m/2d}$
%\end{itemize}
%For even $d$:
%\begin{itemize}
%	\item $C_{G^\ast}(s)=C_{k}(q) C_{n-k}(q)$, $L^\ast=C_m(q) (q^{d_0}+1)^{n-m/d_0}$, $C_{L^\ast}(s)=C_l(q) C_{m-l}(q) (q^{d_0}+1)^{n-m/d_0}$
%\item If $d_0$ is odd: $C_{\G^\ast}(s)=C_{n/2}(q^2)$, $L^\ast=C_m(q) A_1(q^{d_0})^{n-m/2d_0}(q^{d_0}+1)^{n-m/2d_0}$, $C_{L^\ast}(s)=C_{m/2}(q^2) (q^{d}-1)^{n-m/2d_0}$
%\item If $d_0$ is even:  $L^\ast=C_{m}(q) (q^{d_0}+1)^{n-m/2d_0}$ and $C_{L^\ast}(s)=C_{m/2}(q^2) (q^{d_0}+1)^{n-m/2d_0}$.
%\end{itemize}

\begin{proof}
	
	Let $(\Levi,\lambda)$ be a $d$-cuspidal pair as in the statement of the lemma.  According to \cite[Remark 2.2]{Cabanesgroup}, this $d$-cuspidal pair corresponds via Jordan decomposition to the $\C_{G^\ast}(s)$-orbit of a $d$-cuspidal unipotent pair $(\C^\circ_{\Levi^{\ast}}(s),\lambda_s)$ of $\C^\circ_{\G^\ast}(s)$ and $\Levi^\ast=\C_{\G^\ast}(\Z^\circ(\C^\circ_{\Levi^\ast}(s))_{\Phi_d})$.

We therefore first determine the possible rational structures of the $d$-split Levi subgroups $\C^\circ_{\Levi^{\ast}}(s)$ of $\C^\circ_{\G^\ast}(s)$. If $\C^\circ_{G^\ast}(s)$ is of type $C_k(q) C_{n-k}(q)$ then the structure indicated in Table \ref{table} follows directly from \cite[Example 3.5.29]{GeckMalle}.
Let's therefore consider the case when $\C^\circ_{G^\ast}(s)$ is of type of $C_{n/2}(q^2)$. Recall that $d$ is the order of $q$ modulo $\ell$ and $d_0$ is the order of $q^2$ modulo $\ell$. We now apply \cite[Example 3.5.29]{GeckMalle} together with Remark \ref{scalar descent}. We conclude that if $d_0$ is odd, the various $d$-split Levi subgroups have the form 
$$C_{m/2}(q^2) ((q^2)^{d_0}-1)^{n-m/2d_0}$$ for some integer $m$ while if $d_0$ is even the Levi subgroups have rational form 
$$C_{m/2}(q^2) ((q^2)^{d_0/2}+1)^{n-m/2d_0}=C_{m/2}(q^2) (q^{d_0}+1)^{n-m/2d_0}$$
for some $m$.

We now compute the structure of $(\Levi,F)$ and we concentrate on case 2 since the computations in the other cases are similar.
In \cite[Example 3.5.15]{GeckMalle} the structure of a $d$-split Levi subgroup is described. In particular $L$ has rational type
$$A_{n_1}( \varepsilon q^{d_0}) \dots A_{n_x}( \varepsilon q^{d_0}) B_y(q) (q^{d_0}-\varepsilon)^{x}$$
%$$L_0=GL_{n_1}( \varepsilon q^{d_0}) \times \dots \times GL_{n_s}(\varepsilon q^{d_0}) \times H,$$
%where $H$ is a special orthogonal group of smaller rank.
%In particular, $\Levi$ has the same polynomial order as $L_0$.
for some integers $x,y,n_1,\dots,n_x$ which satisfy $d_0x+y+ \sum_{i=1}^x d_0 n_i=n$. Since $\Z^\circ(L)$ has order $(q^d-1)^x$ and $\Z^\circ(\C^\circ_{\Levi^\ast}(s))_{\Phi_d}=\Z^\circ(\Levi^\ast)_{\Phi_d}$, we must have $x=\frac{n-m}{2d_0}$.

The natural inclusion $[\Levi,\Levi] \hookrightarrow \Levi$ induces a surjective dual map $\Levi^\ast \twoheadrightarrow [\Levi,\Levi]^\ast$ and we denote by $s_0$ the image of $s$ under this map. Then any character $\lambda_0 \in \Irr(L_0 \mid \lambda)$ where $L_0:=[\Levi,\Levi]^F$ lies in the Lusztig series associated to $s_0$ and is $d$-cuspidal. We now consider the projection of $\lambda_0$ to the various rational components of $L_0$.

In particular, the projection of $s_0$ onto the simple component of $([\Levi,\Levi]^\ast)^F$ of type $C_y(q)$ is an involution with centralizer containing $C_{m/2}(q^2)$. Hence, we must have $y = m$, see \cite[Table 4.3.1]{GLS}.

Note that a $d$-cuspidal character of $\SL_{n_i+1}(\varepsilon q^{d_0})$ with respect to its $\mathbb{F}_{q}$-structure corresponds to a $1$-cuspidal character (resp. $2$-cuspidal if $\varepsilon=-1$) of $\SL_{n_i+1}(\varepsilon q^{d_0})$ with respect to its $\mathbb{F}_{q^{d_0}}$-structure, see Remark \ref{scalar descent}. Now, the connected centralizer of the semisimple element associated to such a cuspidal (resp. $2$-cuspidal) character of $\SL_{n_i+1}(\varepsilon q^{d_0})$ is a Coxeter torus, see the proof of \cite[Proposition 2.7]{BroughC}. Moreover, since this semisimple element has order at most $2$, by \cite[Table 4.3.1]{GLS} it follows that $n_i \leq 1$. We can now compare this with the polynomial order of $\C_{\Levi^\ast}(s)$ and it follows that $n_i = 1$ for all $i$. In particular, $L$ has the structure as described in Table \ref{table}.
%Moreover, the centralizer of a semisimple element will at least contribute a maximal torus of each of the $GL_{n_i}(q^d)$. Also note that the equality of centralizers is generic, i.e. if it holds for $q$ then it holds for all $q^k$ with $k \equiv 1 \mod 2d$. Hence, $n_i=2$ for all $i$. This shows that the structure is as described.  
\end{proof}

For $d=1$, the results of the previous lemma were already obtained in \cite[Table 4.1]{Taylor3}.

We observe that in cases 1 and 3, the appearing Levi subgroups were already considered in Cabanes--Späth \cite{TypeB}, see Remark \ref{modifaction}. In particular, we can in the following focus on case 2.
\section{The local condition for groups of type $B$}

\subsection{Root system and Weyl group}
	Let ${\bf T}$ be an $F$-stable maximal torus and ${\bf B}$ be an $F$-stable Borel subgroup of ${\bf G}$ with ${\bf T}\subseteq {\bf B}$.
%	
%	and ${\bf  N}:={\rm N}_{\G}({\bf T})$.
Set $\Phi\subseteq {\rm Hom}({\bf T},\mathbb{F}^\times)$ and $\Phi^+\supseteq \Delta$ to be the corresponding set of roots, positive and simple roots associated with the choice of ${\bf B}$.
The Chevalley generators ${\bf x}_\alpha(t)$, ${\bf n}_\alpha(t')$ and ${\bf h}_\alpha(t')$, for $\alpha\in \Phi$, $t,t'\in \mathbb{F}$  with $t'\ne 0$, together with the Chevalley relations describe the group structure of $\G$, see \cite[Theorem 1.12.1]{GLS}.
In particular, ${\bf n}_\alpha(t)={\bf x}_\alpha(t){\bf x}_{-\alpha}(-t^{-1}){\bf x}_\alpha(t)\in {\rm N}_{\G}({\bf T})$ and ${\bf h}_{\alpha}(t)={\bf n}_\alpha(t){\bf n}_\alpha(1)^{-1}\in {\bf T}$ for $t\in \mathbb{F}^\times$. We let $F_q: \G \to \G$ be the Frobenius endomorphism with $F_q(\x_\alpha(t))=\x_{\alpha}(t^q)$ for $\alpha \in \Delta$. Recall Tits' extended Weyl group $\V:=\langle {\bf n}_\alpha(1)\mid \alpha\in \Phi\rangle \leq {\rm N}_{\G}({\bf T})$ and $\bH:={\bf T}\cap \V$.

%
%\begin{notation}\label{n_typB}
The root system $\Phi$ of $\bG$ and its system of simple roots $\Delta =\{\alpha_1,\ldots,\alpha_n\}$ are given as follows (see for instance \cite[1.8.8]{GLS} with a slightly different convention). Letting $(e_1,\dots ,e_n)$ be the standard orthonormal basis of $\mathbb{R}^n$ with its euclidean scalar product, one takes 
$$\alpha_1= e_1, \alpha_2=e_2-e_1, \dots , \alpha_n=e_n-e_{n-1}.$$
We identify the Weyl group $\W$ of $\bG$, a Coxeter group of type $B_n$ with the subgroup of bijections $\sigma$ on $\{ \pm 1, \ldots \pm n\} $ such that $\sigma(-i)=-\sigma(i)$ for all $1\leq i\leq n$.
% see also \cite[\S 5]{S10a}.
 This group is denoted by $\mathfrak{S}_{\pm n}$. 
Via the natural identification of $S_n$ with a quotient of $\mathfrak{S}_{\pm n}$ and the identification of $\mathfrak{S}_{\pm n}$ with $\W$, the map $\rho: {\rm N}_{\G}({\bf T}) \to \W$ induces an epimorphism $\overline{\rho}\colon \V\longrightarrow \mathfrak{S}_n$.
(see \cite[Definition 4.1]{TypeB}).
In $\mathfrak{S}_{\pm n}$ the set
$$ \left \{ 
\sigma \in \mathfrak{S}_{\pm n} \, \, \big| |\{1,\ldots, n\} \cap \sigma^{-1}(\{-1,\ldots, -n\})| \text{ is even } \right \} $$
forms a normal subgroup with index $2$, naturally isomorphic to a Coxeter group of type $D_n$. We denote by $\W_{D}$ the associated subgroup of $\W$ and $\bN_c:=\rho^{-1}(\W_{D})$. 
%\end{notation}

\subsection{Sylow $d$-tori}\label{Sylow}

Take $n'\leq n$ to be maximal such that $d_0\mid n'$.
Set $a:=\frac{n'}{d_0}$, $v_0:={\bf n}_{\alpha_1}(1)\cdots {\bf n}_{\alpha_{n'}}(1)\in \V$ and $v:=v_0^{\frac{2n'}{d}}$. Additionally set $$w_0:=\rho(v_0)=(1,2,\dots, n',-1,\dots,-n') \text{ and } w=(w_0)^{\frac{2n'}{d}}.$$
Finally, set $v':=(v_0)^a$ and 
$$w':=\rho(v')=(w_0)^a=\prod_{i=1}^a (i,a+i,2a+i,\dots, (d_0-1)a+i,-1,\dots).$$
Note that $\C_{\W}(w)=\C_{\W}(w')$. For our computations it is sometimes convenient to work with $w'$ instead of $w$.
According to \cite[Remark 3.2]{BS}, the Sylow $d$-torus of $({\bf T},vF_q)$ is one of $({\bf G},vF_q)$. In particular, any $d$-split Levi subgroup of $(\G,vF_q)$ is up to $\G^{vF_q}$-conjugation the centralizer of an $vF_q$-stable subtorus of $\T$.

\subsection{Structure of Levi subgroups}\label{structure Levi}
Recall that we are only interested in $d$-split Levi subgroups whose fixed point subgroup under $vF_q$ is of type $A_1(q^{d_0})^{\frac{n-m}{2d_0}} B_m(q)$ for some $m$ with $2 d_0 \mid (n-m)$. Therefore, we first replace $v$ by an alternate twist which is chosen to be "nice" with respect to the considered Levi subgroup. For this purpose set $$l:=n-m, \, t_l:=\frac{l}{2d_0} \text{ and }a_l:=2 t_l .$$
The following constructions will depend on the parameter $l$.

There is some $h\in V$ such that $$((w')^{\rho(h)})^{-1}\left( \prod_{i=1}^{a_l}(i,a_l+i,\dots, (d_0-1)a_l+i,-i,\dots)\right)\in \mathfrak{S}_{\pm}(\{l+1,\dots ,n\})$$
and the Sylow $d$-torus of $({\bf T},v^hF_q)$ is one of $({\bf G},v^hF_q)$. In particular,
%hence also a suitable twist. In particular, 
$\rho(\C_{\V}(v^h))=\C_{\W}(w^{\rho(h)})$.

Let $\Levi$ be a $d$-split Levi subgroup for $({\bf G},v^hF_q)$ whose fixed point subgroup is of type $A_1(q^{d_0})^{t_l} B_m(q)$.
Then the root system $\Phi_\Levi$ of $\Levi$ must be $w^{\rho(h)}$-stable.
Hence after suitable conjugation in ${\bf G}^{v^hF_q}$, and by \cite[Remark 3.12]{BroughC}, it can be assumed that 
$$ 
\Phi_{\bf L}=\big\{ \pm e_i,\pm e_i\pm e_j\mid l+1\leq  i,j\leq n \big\} \sqcup \bigsqcup_{i=1}^{t_l} \big\{\pm(e_{2i}-e_{2i-1}) \big\}. 
$$
We denote $\V_{l}:=\langle \n_{\alpha}(1) \mid \alpha \in \Phi \cap \langle e_{i} \mid i \leq l +1 \rangle \rangle$. One sees that $\V_{l}$ is the extended Weyl group associated to the root system $ \Phi_l:=\Phi \cap \langle e_{i} \mid i \leq l +1 \rangle \rangle$ of type $B_{l}$. 
Set $v_{l,0}:={\bf n}_{\alpha_1}(1)\cdots {\bf n}_{\alpha_{l}}(1)\in \V_{l}$, whose image is $w_{l,0}=(1,2,\dots, l,-1,\dots)$.
Additionally, set $v_l':=(v_{l,0})^{a_l}$, $v_l:=(v_{l,0})^{\frac{2l}{d}}$, $w_{l}':=\rho(v_{l}')$ and $w_l:=\rho(v_{l})$.
In particular, $w_l'=\prod_{i=1}^{a_l}(i,a_l+i,\dots, (d_0-1)a_l+i,-i\dots)$.
Then $v_l^{-1}v^h\in \Levi$ and thus the Levi subgroups $\Levi^{v^hF_q}$ are $\Levi^{v_{l}F_q}$ isomorphic by \cite[Lemma 4.2]{BroughC}.
We will therefore in the following consider the twist $F:=v_l F_q$.

\subsection{Blocks and diagonal automorphism}

We can now use the explicit description of the root system of the Levi subgroup $\Levi$ to deduce some block-theoretic properties.

\begin{lemma}\label{blocktheory}
	Suppose that we are in case 2 of Table \ref{table}.
	%The centralizer $\C_{L^\ast}(s)$ is disconnected of rational type $C_{m/2}(q^2) (\Phi_1 \Phi_2)^{(n-m)/2}.2$ and the group $A(s)$ acts as the field automorphism $F_q$ on the connected component. In particular, by Jordan decomposition every character of $\mathcal{E}(L,s)$ has a $\tilde{L}$-orbit of size $2$.
	Then the blocks of $\mathcal{E}_\ell(G,s)$ are not $\tilde{G}$-stable.
\end{lemma}

\begin{proof}
	Keep the notation of the proof of Lemma \ref{structure}.
	By Corollary \ref{coro} it suffices to show that $(\Levi,\lambda)$ is not $\tilde{L}$-stable. We first claim that $\Z(\Levi)$ is disconnected. For this observe that it suffices to prove similar to the proof of \cite[Lemma 3.3]{KKL} that $X(\T)/ \mathbb{Z}\Phi_\Levi$ has non-trivial $p'$-torsion. We observe that $$\mathbb{Z}\Phi_\Levi=\langle \alpha_1,\dots,\alpha_m,\alpha_{m+2},\dots,\alpha_{n-2},\alpha_{n} \rangle= \langle e_1,\dots,e_m,e_{m+1}-e_{m+2},\dots,e_{n-3}-e_{n-2},e_{n-1}-e_n \rangle.$$ By the proof of \cite[Lemma 3.3]{KKL} it follows that $\frac{1}{2} (\sum_{i=1}^m e_i + \sum_{i=1}^{(n-m)/2} e_{m+2i-1}-e_{m+2i}) \in X(\T) \setminus \mathbb{Z}\Phi$ but $\sum_{i=1}^m e_i + \sum_{i=1}^{(n-m)/2} e_{m+2i-1}-e_{m+2i} \in \mathbb{Z} \Phi_\Levi$. Hence, $\Z(\Levi)$ is disconnected.
	
	Since $\tilde{L}/L_0$ is abelian and $\Z(\Levi)$ is disconnected it suffices to show that any $\lambda_0 \in \Irr(L_0 \mid \lambda)$ is not stable under any diagonal automorphism of $L_0$. Now, $L_0^\ast \cong B_m(q) \times \mathrm{PGL}_2(q^{d_0})^{a_l/2}$ while $\C_{L_0^\ast}(s_0) \cong C_{m/2}(q^2) \times C_{q^{d_0}-\varepsilon}^{a_l/2}$. By \cite[Table 4.3.1]{GLS} this proves that the projection of $s_0$ onto each factor of $L_0^\ast$ is an involution with disconnected centralizer. By \cite[Proposition 4.4]{MalleHZ} and Jordan decomposition \cite[Theorem 2.6.22]{GeckMalle} we therefore deduce that the character $\lambda_0$ is not invariant under any diagonal automorphism.
	%By \cite[Proposition 4.4]{MalleHZ} the $\C^\circ_{G^\ast}(s)$-orbit of the unipotent $d$-cuspidal character $(\C^\circ_{\Levi^\ast}(s),\lambda_s)$ is $\C_{G^\ast}(s)$-stable. From this it follows from Jordan decomposition \cite[Theorem 2.6.22]{GeckMalle} that the $\tilde{G}$-orbit of $(\Levi,\lambda)$ has size $2$. We deduce by \cite[Theorem A]{Cabanesgroup} that every block in $\mathcal{E}_\ell(G,s)$ has a $\tilde{G}$-orbit of size $2$. 
\end{proof}

\begin{remark}\label{rmk}
	Suppose that $(\Levi,\lambda)$ is a $d$-cuspidal pair corresponding to a block in case 2 of Table \ref{table}. Then $\tilde{L}_\lambda=L\Z(\tilde{G})$ by Lemma \ref{blocktheory}. By the proof of Proposition \ref{star}, the quotient $\tilde{L}_{\theta}/\tilde{L}_{\Lambda(\theta)}$ is in bijection with $\Irr(W(\theta)/W(\tilde{\theta}))$ and $\tilde{L}_\theta \leq \tilde{L}_\lambda$ for every $\theta \in \Irr(L,b_L(\lambda))$. We conclude that $W(\theta)=W(\tilde{\theta})$ for every $\theta \in \Irr(L,b_L(\lambda))$. 
\end{remark}

\subsection{The structure of $N/L$}\label{relative Weyl}

Recall that we work with the Frobenius endomorphism $F:=v_{l}F_q$ and the $d$-split Levi subgroup $\Levi$ of $(\G,F)$ with $w_l$-stable root system $\Phi_\Levi$ as given above. Recall that $$\W_{\bf G}({\bf L}):={\rm N}_{\bf G}(\Levi)/\Levi\cong {\rm N}_{W_{\bf G}}(\W_{\Levi})/\W_\Levi \text{ and }{\rm N}_{\bf G}(\Levi)^{v_l F_q}/\Levi^{v_l F_q}\cong {\rm C}_{\W_{\bf G}({\bf L})}(w_l \W_\Levi).$$

For $1\leq i\leq {t_l}$ set $$w_{l,i}':=(2i-1,a_l+2i-1,\dots, (d_0-1)a_l+2i-1,-(2i+1),\dots)(2i,a_l+2i,\dots, (d_0-1)a_l+2i,-(2i),\dots)$$ so that $w_l'=\prod_{i=1}^{t_l} w_{l,i}'$.
Additionally, for $1\leq i\leq {t_l}-1$, set $$\tau_i:=\prod_{k=0}^{d_0-1}  \big( (2i-1,2i+1)(2i, 2i+2)(-(2i-1),-(2i+1))(-(2i),-(2i+2)) \big)^{(w_l')^k}.$$
Then $$W':=\big( \prod_{i=1}^{t_l}\langle w_{l,i}'\rangle\big)\rtimes \langle \tau_i\mid 1\leq i\leq {t_l}-1\rangle\leq {\rm C}_{\W_{l}}(w_l)$$ whose image in $\W_{\bf G}({\bf L})$, that is $W'\W_\Levi/\W_\Levi$, is ${\rm C}_{\W_{\bf G}({\bf L})}(w_l \W_\Levi)\cong C_{2d_0}\wr \mathfrak{S}_{t_l}$ (as in \cite[Section 4]{BroughC}).
Our aim in the next sections is to construct a supplement $V' \leq  {\rm N}_{\bf G}(\Levi)$
%$V_0\leq (V_{B_{l}})^{v_m}\cap {\rm N}_{\bf G}(\Levi)$
 with $\rho(V')=W'$. By construction, the subgroup $W'$ is contained in $\W_{l}$. Therefore, we will construct the subgroup $V'$ inside $\V_{l}^{v_l F_q}$. Note that $v_l$ is by construction a Sylow $d$-twist and $d$ is a regular number for the root system $B_{l}$, i.e. $d \mid 2l$. Therefore, we are precisely in the situation already considered in \cite[Section 5.A]{TypeB}.

\section{The extended Weyl group and $d$-twists}

In this section, we recall the main construction from \cite[Section 5]{TypeB} which we will then modify later.
We keep the notation of the previous section. In particular, we let $\V_{l}$ be the extended Weyl group associated to a root sytem $\Phi_l$ of type $B_l$. Let $\rho: \V_l \to \W_l$ be the natural epimorphism to the Weyl group with kernel $\bH_l$. By construction in the last section, the integer $d$ divides $2l$ and thus $a_l:=\frac{l}{d_0}$ is an integer. We let $v_l \in \V_l$ be a Sylow $d$-twist as in \ref{structure Levi} (see also \cite[Section 5.A]{TypeB}) and denote $H_l:=\bH_l^{v_lF_q}$ and $V_l:=\V_l^{v_l F_q}$.

For a subset $\mathcal{O} \subset \{ 1,\dots,l \}$ we define $\Phi_{\mathcal{O}}:=\Phi_l\cap {\langle\pm e_i\mid i\in \mathcal{O}\rangle }$ and $\V_{\mathcal{O}}:=\langle \n_\alpha (\pm 1)\mid \alpha\in \Phi_{\mathcal{O}} \rangle$ and $\bH_{\mathcal{O}}:=\V_{\mathcal{O}} \cap \bH$.

%{\color{blue} I think we need to rephase all the following in the group $V_{B_{l}}$ using $v_m$ and $a_l$ instead!}

\subsection{The group $H_l$.} 

Set $\oo v_l:=\ov\rho (v_l)$ and for $k \in \{1,\dots,l\}$ let $\mathcal{O}_k$ be the $\oo v_l$-orbit on $\{1,\dots ,l\}$ containing $k$.
Let $\varpi\in\mathbb{F}$ be a fourth root of unity and $$h_0:=\h_{e_1}(-1).$$ For $1\leq k\leq {a}_l$ let $$h_k:=\prod_{i\in \mathcal{O}_k}\h_{e_i}(\varpi).$$ By the Chevalley relations the following equalities hold:
\begin{align}
	h_k ^2= \h_{e_1}(-1)^{|\mathcal{O}_k|}=h_0^{|\mathcal{O}_k|}\label{eq10.1}
	&\text{ and }\\
	h_k^v= \begin{cases} h_k & \text{if } 2\nmid d\\
		h_0h_k& \text{otherwise.} \end{cases}&
\end{align}
By \cite[Section 5.B]{TypeB}, we have
%Note that for any $h\in H$ there exist elements $t_i\in \langle\varpi\rangle$ such that $h=\prod_{i=1}^l \h_{e_i}(t_i)$ and $(\prod t_i)^2=1$. 
%One shows that $h_k\notin H_d$ and $h_k h_{k+1}\in H_d$ for every $1\leq k <a$. Hence $H_d$ is an elementary abelian $2$-group of rank $a$, namely
$$H_l=\langle h_0\rangle \times \langle h_1h_2\rangle \times\cdots\times \langle h_{a_l-1}h_{a_l}\rangle,$$ 
an elementary abelian $2$-group of rank $a_l$.
\subsection{The group $\V_l$ and some of its elements.}\label{elements}
We define $\ov c_1 \in \C_{\W_l}({\rho(v_l)})$ as in \cite[Section 5.B]{TypeB}.
If $2\mid d$ let $$\ov c_1:=
(1,a_l+1,\ldots, a_l (d_0-1)+1, -1, -a_l-1,\ldots, -a_l (d_0-1)-1 )\in \C_{\W_l}({\rho(v_l)}).$$ 
This is the cycle of $\rho(v)$ containing $1$. If $2\nmid d$ let 
\begin{align*}
	\ov c_1':=
	&(\,\,\,1,&\,2a_l+1,&\,\,\,\, 4a_l+1,& \ldots, \,(d-1)a_l+1,& -a_l-1,&\ldots, -(d-2)a_l-1 )&\\
	&(- 1,&- 2a_l-1,&-4a_l-1,& \ldots, -(d-1)a_l-1,& a_l+1,&\ldots, (d-2)a_l+1 )&
	\in \C_{\W_l}(\rho(v)),
\end{align*} 
and set $\ov c_1:=\ov c_1'\prod_{i=0}^{d-1 } (ia_l+1,-ia_l-1)$. 
%
%Note that $\ov c_1'$ is the pair of cycles of $\rho(v)$ containing $\pm 1$. 	

Since $\rho(\V_l)=\W_l$ there exists some $c_1\in \V_l$ with $\rho(c_1)=\overline{c}_1$. As in \cite[Section 5.B]{TypeB} one can even choose $c_1$ such that $c_1\in(\V_{\mathcal{O}_1})^{v_lF_q}$ and $ (\V_{\mathcal{O}_1}\cap V_l)\C_{\bH_{\mathcal{O}_1}}(v_l)= \langle h_1,h_0,c_1 \rangle$.

%In the next step we construct $c_1\in(V_{\mathcal{O}_1})^{vF_q}$.
%By definition $\n_\alpha(\pm 1)^2\in H$ and hence $c_1$ can be chosen to be contained in $\T V_{\mathcal{O}_1}$. 
%Let $H_{1}:=\langle \h_\alpha (\varpi)\mid \alpha\in \Phi_1\rangle$, 
%$\Phi_1':=\Phi\cap {\langle\pm e_i\mid i\notin \mathcal{O}_1\rangle }$, $H_{1'}:=\langle \h_\alpha (\varpi)\mid \alpha\in \Phi_1'\rangle$ and $V_{\mathcal{O}_1}:=\langle \n_\alpha (\varpi)\mid \alpha\in \Phi_1\rangle$. Then $c_1\in H_{1'}V_{\mathcal{O}_1}$. Since $V_{\mathcal{O}_1}\cap H_{1'}=\langle h_{e_1}(-1)\rangle$ by Chevalley relations we see that $c_1=c_1' h$ for some $c_1'\in V_{\mathcal{O}_1}$ and $h\in H_{1'}$. Let $J\subseteq \{1,\ldots , l\} \setminus \mathcal{O}_1$ with $h=\prod_{i\in J} h_{e_i}(\varpi)$. The equality $c_1^v=c_1$ implies $[h,v]=[c_1',v]\in \{\h_{e_1}(\pm 1)\} $. Since $c_1\in V$ either $c_1',h\in V$ or $c_1',h\notin V$. As $[h,v]\in \{\h_{e_1}(\pm 1)\} $ we see that $J$ is $\oo{ v}$-stable, hence a union of orbits $\mathcal{O}_k$. If $d_0=d$ this implies $[h,v]=1$ and if $2\nmid d_0$, $[h,v]=1$ and $h\in H$ are equivalent. Going through all possible cases one sees that either $c_1'$ or $c_1'h_1$ is contained in $V_{\mathcal{O}_1}\cap V_d=V^{vF_q}$. In the following we denote this element by $c_1$. 
%The group $ (V_{\mathcal{O}_1}\cap V_d)\C_{H_1}(v)$ is generated by $h_1$, $h_0$ and $c_1$, i.e., 
%$ (V_{\mathcal{O}_1}\cap V_d)\C_{H_1}(v)= \langle h_1,h_0,c_1 \rangle$.

%Recall $v_k:=\n_{\alpha_k}(1)\in \bG$.
As in \cite[Section 5.B]{TypeB}, let $p_k:=\prod_{i=0}^{d_0-1}\n_{\alpha_{k+1}}(1)^{v_l^i}$ for $1\leq k \leq a_l-1$, which satisfy
$$
\rho(p_k)(\mathcal{O}_k)=\mathcal{O}_{k+1} \text{ and }\rho(p_k)(\mathcal{O}_{k+1})=\mathcal{O}_{k},$$ 
as well as $p_k\in V_l$ and $p_k \in \V_{\mathcal{O}_k\cup\mathcal{O}_{k+1}}$.
%, where $V_{\mathcal{O}_k\cup\mathcal{O}_{k+1}}:=\langle \n_\alpha (\pm 1)\mid \alpha\in \Phi_{k,k+1}\rangle$ with $\Phi_{k,k+1}:=\Phi\cap {\langle\pm e_i\mid i\in \mathcal{O}_k\cup \mathcal{O}_{k+1}\rangle }$.
The elements $p_k$ satisfy the type $A$ braid relations and we have $p_k^2=h_k h_{k+1} h_0^{|\mathcal{O}_k|}\in H_l$. 

%
%\begin{align}\label{eq10.pk2}
%	\begin{split}
%		p_k^2&=\prod_{i=0}^{d_0-1}({v_k}^{v^i})^2=
%		\prod_{i=0}^{d_0-1} \h_{e_{\oo{v}^i(k)}-e_{\oo{v}^i(k+1)}} (-1) \\
%		&=\prod_{j\in \mathcal{O}_k}\h_{2e_j}(\varpi)\left (\prod_{j\in \mathcal{O}_{k+1}}\h_{2e_j}(-\varpi)\right )=
%		h_k h_{k+1} h_0^{|\mathcal{O}_k|}\in H_d. 
%	\end{split}
%\end{align}
%Note that since the elements $v_1,\dots, v_{a-1}$ satisfy the braid relations this applies also to $p_k$.
For $2\leq k \leq a_l$ let $c_k:=(c_1)^{p_1 \cdots p_{k-1}}$.
%From their definitions one shows that these elements satisfy the following equations: 
%For every $0\leq i\leq a$, $1\leq j \leq a$ and $1\leq k <a$ we have 
%\begin{align}
%	h_i^{c_j}&=\begin{cases} h_i &\text{if } i\neq j,\\
%		h_i h_0 & \text{if } i=j,	\end{cases} \label{eq10.3hicj}
%\end{align}
%and also 
%\begin{align}\label{eq10.4}
%	h_i^{p_k}&=\begin{cases} h_i &\text{if } i\notin \{k,k+1\},\\
%		h_{i+1} & \text{if } i=k,\\
%		h_{i-1} & \text{if } i=k+1.	\end{cases}
%\end{align}
%Further by the Chevalley relations 
%\begin{align}\label{eq10.5} [c_i,c_j]&=h_0 	\end{align} 
%for every $1\leq i,j\leq a$, see also \cite[2.1.7(c), proof of Lem.~10.1.5]{S07}.
Since $\rho(\V_l)=\langle \ov c_i \mid 1\leq i\leq a_l \rangle \rtimes \mathfrak{S}_{a_l}$ we see that the elements $c_i$ ($1\leq i \leq a_l$) together with the elements $p_k$ ($1\leq k < a_l$) generate the group $\V_l$.

\section{A supplement of the relative Weyl group}

We will now modify the construction in the last section to construct a supplement of the relative Weyl group.
%Recall that $a_l:=l/d_0$ and $t:=a_m/2$.
%In the previous section, we recalled the construction of the group $H_d \cong C_{2}^{a/2}$ and the group $V_d$ with $V_d/H_d \cong C_{2 d_0} \wr S_{a/2}$
%{\color{blue} this will be $a_m/2$}. The group $V_d$ (and thus also $H_d={\rm ker}(\rho|_{V_d})$) is too big. 

First, there is a set theoretic map
%For $1\leq i\leq a_l-1$ set $p_i:=\prod_{k=0}^{d_0-1}{{\bf n}_{\alpha_{i+1}}(1)}^{v_l^k}$, where we recall that $\alpha_{i+1}=e_{i+1}-e_i$.
%Thus there is a map 
$$\begin{array}{ccc}
	\iota_1:\langle {{\bf n}_{\alpha_{2}}(1)},\dots, {{\bf n}_{\alpha_{a_l}}(1)}\rangle & \rightarrow & \langle p_1,\dots, p_{a_l-1}\rangle,\\
	n & \mapsto & \prod_{k=0}^{d_0-1} n^{v_l^k},
\end{array}
$$
and by definition $\iota_1(\n_{\alpha_k}(1))=p_{k-1}$ for $2 \leq k \leq a_l$. 
Next we define two elements 
$$
g_1:=\prod_{i=1}^{t_l} p_i^{p_{i+1}\dots p_{2i-2}},\indent \text{ and } \indent g_2:= \prod_{i=1}^{t_l} p_i^{p_{i+1}\dots p_{2i-1}},
$$
which yields another map
$$
\begin{array}{ccc}
	\iota_2:\langle p_1,\dots, p_{{t_l}-1}\rangle & \rightarrow &  \langle p_1,\dots,p_{a_l} \rangle,\\
	p & \mapsto &  p^{g_1}p^{g_2}.\\
\end{array}
$$

\begin{lemma}\label{homomorphism}
	The maps $\iota_1$ and $\iota_2$ are injective group homomorphisms. 
	Moreover, for $1\leq i\leq {t_l}-1$, the elements $p_i':=\iota_2(p_i)$ satisfy the type $A$ Braid relations and $(p_i')^2=h_{2i-1}h_{2i}h_{2i+1}h_{2i+2}$.
\end{lemma}
\begin{proof}
	
	%The maps $\iota_1$ and $\iota_2$ are defined within extended Weyl group of a root system of type $A_{l}$, that is $V_{A_{l}}=\langle {\bf n}_{\alpha}(1)\mid \alpha=e_i-e_j,\text{ with } 1\leq i,j\leq l\rangle$.
	%In particular, ${\bf n}_{\alpha_{i+1}}(1)^{(v_l')^k}\in \langle {\bf n}_{e_{ka_l+i+1}-e_{ka_l+i}}(1) \rangle$ and the sets $ \{ka_l+i+1,ka_l+i\}$ for $1\leq i\leq a_l-1$ and $0\leq k\leq d_0-1$ are disjoint.
	%Hence by {\color{red} [Brough Corollary 3.3]}, the map $\iota_1$ is an injective homomorphism.
	To show that $\iota_1$ is a homomorphism it suffices to check that  $\iota_1(\n_{\alpha_i}(1) \n_{\alpha_j}(1))=\iota_1(\n_{\alpha_i}(1)) \iota_1(\n_{\alpha_j}(1))$ for $2 \leq i,j \leq l$.
	
	First note for $2\leq i\leq a_l$ and $0\leq k\leq d_0-1$ that ${{\bf n}_{\alpha_{i}}(1)}^{(v_l')^{k}}\in \langle {\bf n}_{e_{ka_l+i+1}-e_{ka_l+i}}(1) \rangle$.
	As $a_l\geq 2$ we have $e_{ka_l+i+1}-e_{ka_l+i}\perp e_{k'a_l+i+1}-e_{k'a_l+i}$ for $k \neq k'$.
	Moreover, as both are long roots, it follows that $[{\bf n}_{e_{ka_l+i+1}-e_{ka_l+i}}(1),{\bf n}_{e_{k'a_l+i+1}-e_{k'a_l+i}}(1)]=1$.
	In particular, we have $p_i\in \prod_{k=0}^{d_0-1}\langle {\bf n}_{ e_{ka_l+i+1}-e_{ka_l+i}}(1)\rangle$.
	
	For $0\leq k< k'\leq d_0-1$, then $e_{ka_l+i+1}-e_{ka_l+i}\perp e_{k'a_l+j+1}-e_{k'a_l+j} $ and both are long roots.
	Hence as before
	${{\bf n}_{\alpha_{i}}(1)}^{(v_l')^{k'}}{{\bf n}_{\alpha_{j}}(1)}^{(v_l')^k}={{\bf n}_{\alpha_{j}}(1)}^{(v_l')^k}{{\bf n}_{\alpha_{i}}(1)}^{(v_l')^{k'}}$.
	Thus $\iota_1$ is a homomorphism. Moreover, $\iota_1$ is injective as $\prod_{k=0}^{d_0-1}\langle {\bf n}_{ e_{ka_l+i+1}-e_{ka_l+i}}(1)\rangle \cap \prod_{k=0}^{d_0-1}\langle {\bf n}_{ e_{ka_l+j+1}-e_{ka_l+j}}(1)\rangle=1$ for $i \neq j$.
	
	Observe by construction that 
	$$
	\rho(g_1)=\prod_{i=1}^{t_l}\left( \prod_{k=0}^{d_0-1}(ka_l+i,ka_l+2i-1)\right),\indent \text{ and } \indent \rho(g_2)= \prod_{i=1}^{t_l} \left( \prod_{k=0}^{d_0-1}(ka_l+i,ka_l+2i)\right)
	$$
	and so for $1\leq i\leq {t_l}-1$, 
	$$
	p_i^{g_1}\in \langle {\bf n}_{\alpha}(1)\mid \alpha=e_{ka_l+2i-1}-e_{ka_l+2i+1} \text{ with } 0\leq k\leq d_0-1\rangle
	$$
	and 
	$$
	p_i^{g_2}\in \langle {\bf n}_{\alpha}(1)\mid \alpha=e_{ka_l+2i}-e_{ka_l+2i+2} \text{ with } 0\leq k\leq d_0-1\rangle.
	$$
	As in the previous case, the long roots $e_{ka_l+2i-1}-e_{ka_l+2i+1}$ and $e_{k'am+2j}-e_{k'a_l+2j+2}$ orthogonal and it follows that $[p_i^{g_1},p_j^{g_2}]=1$.
	It follows from this that $\iota_2$ is a homomorphism.
	
	To compute the formula for $\iota_2(p_i)^2$ we know that $p_i^2=h_ih_{i+1}h_0^{d_0}$ by \ref{elements} and thus $\iota_2(p_i)^2=h_{2i-1}h_{2i}h_{2i+1}h_{2i+2}$. Since $\iota_2$ is injective modulo $H_l$ and $H_l \cap \langle p_1,\dots p_{t_l-1} \rangle =\langle p_1^1,\dots p^2_{t_l-1}\rangle$ it follows from this formula that $\iota_2$ itself must be injective.
%	Or
%observe that
%	$\iota_2(p_i)^2=\iota_2(\iota_1({\bf n}_{\alpha_{i+1}}(1)))^2$ and ${\bf n}_{\alpha_{i+1}}(1)^2=\h_{\alpha_{i+1}}(-1)$.
\end{proof}

\begin{corollary}
	For $c_1\in (V_{\mathcal{O}_1})^{v_lF_q}$ as constructed above and $V':=\langle c_1 c_1^{p_1},p_i'\mid 1\leq i\leq {t_l}-1\rangle$, we have ${\rm N}_{G}(\Levi)=LV'$.
\end{corollary}
\begin{proof}
	Note that by construction in \cite[Section 5.B]{TypeB}, the elements $c_1$ and $p_i$ with $1\leq i\leq a_l-1$ all lie in $V_l$.
	Moreover $\rho(c_1  c_1^{p_1})=w_{l,1}'$ and $\rho(p_i')=\tau_i$ as defined in \ref{relative Weyl}.
	Hence $LV'\leq {\rm N}_{G}(\Levi)$ with $LV'/L = {\rm N}_{G}(\Levi)/L$ showing the required equality.
\end{proof}

Mimicking the definition in \cite[Section 5.B]{TypeB}, for $1\leq i\leq t_l$ we define $c_i':=(c_1  c_1^{p_1})^{p_1'p_2'\cdots p_{i-1}'}$, which satisfy $\rho(c_i')=w_{l,i}'$.
Additionally, set $C':=\langle c_i'\mid 1\leq i\leq t_l\rangle$ and $P':=\langle p_i'\mid 1\leq i\leq t_l-1\rangle$.

Then $c_i'\in \langle {\bf n}_{\alpha}(1)\mid \alpha\in \Phi\cap \langle \pm e_k\mid k \in \mathcal{O}_{2i-1} \cup \mathcal{O}_{2i} \rangle \rangle=\V_{\mathcal{O}_{2i-1} \cup \mathcal{O}_{2i}} $. We claim that

$$ (c_i')^{p_j'}= \left\{
\begin{array}{lr}
	c_i' &\text{if } j\notin \{i-1,i\},\\
	c_{i+1}' & \text{if } j=i,\\
	c_{i-1}' & \text{if } j=i-1.\\
\end{array}
\right.
$$
The case $j \notin \{i-1,i\}$ is immediate, while when $j=i-1$ there is $h\in H_l$ such that $p_1'\cdots p_{i-2}'(p_{i-1}')^2=hp_1'\cdots p_{i-2}'$ and by construction $[c_1',H_l]=1$, hence $(c_i')^{p_{i-1}'}=(c_1')^{p_1'\cdots p_{i-2}'}=c_{i-1}'$.

\begin{lemma}\label{abelian}
	The group $C'$ is abelian and is the central product of the groups $\langle c_i' \rangle$, $i=1,\dots, t_l$, over $\langle h_0\rangle$, where $(c_i')^{2d_0}=h_0$.
\end{lemma}
\begin{proof}
	Recall from \cite[Section 5.B]{TypeB} that $c_1^{2d_0}\in\langle h_0\rangle$.
%	(or Lucas argument, by Equation (4) above, i.e. ${}^{c_i} h_i=h_i h_0$, this implies $c_i^{2 d_0} \in \langle h_0 \rangle$).
	As $[c_1,c_2]=h_0$ by \cite[Section 5.B]{TypeB} we see that $(c_1')^{m}=h_0^{1+2+\dots+ m-1} c_1^{m} c_{2}^m$ for $m \in \mathbb{N}$.
	Since $d_0(2d_0-1)$ is odd (as $d_0$ is assumed to be odd) it therefore follows that $(c_1')^{2d_0}=h_0c_{1}^{2d_0} c_2^{2d_0}=h_0$. From this we obtain $(c_i')^{2d_0}=h_0$ for all $i$.
	
	To show $C'$ is abelian consider the commutator $[c_1',c_i']$ for $i > 1$. Observe that $c_i'\in \langle c_j\mid 1\leq j \leq a_l\rangle$ and hence $[c_1,c_i']\in \langle h_0\rangle$.
	Moreover, as $i>1$ and thus $[p_1,c_i']=1$, it follows that $$ 
	[c_1',c_i']=[c_1  c_2,c_i'] = [c_1,c_i']^{c_2}[c_2,c_i']=[c_1,c_i']^{c_2}[c_1,c_i']^{p_1}=1.
	$$
	Let $i<j$. 
	By applying the formula for $(c_i')^{p_j'}$ from above, we deduce that $$[c_i',c_j']=[(c_1')^{p_1'\cdots p_{i-1}'},(c_1')^{p_1'\cdots p_{j-1}'}]=[(c_1')^{p_1'\cdots p_{i-1}'},(c_1')^{p_1'\cdots p_{i}'}]^{p_{i+1}'\cdots p_{j-1}'}=[c_1',c_{i+1}']^{p_1'\cdots p_{i-1}'p_{i+1}'\cdots p_{j-1}'}.$$
	Hence, $[c_i',c_j']=1$.
\end{proof}

%\begin{proof}
%	This is because $[c_i,c_j]=h_0$ for all $i \neq j$. I would assume that $h:=c_i^{2 d_0} \in \C_{H_1}(v)=\langle h_0, h_1 \rangle$. However, $h$ has to be invariant oder the action of all $c_j$'s by construction. By Equation (4) above, i.e. ${}^{c_i} h_i=h_i h_0$, this implies $c_i^{2 d_0} \in \langle h_0 \rangle$. This implies $(c_{i} c_{i+1})^{2 d_0} \in \langle h_0 \rangle$ as well.

%The calculation is: $(c_{i} c_{i+1})^n=h_0^{1+2+\dots n-1} c_i^{n} c_{i+1}^n$, and $1+\dots + n-1=n (n-1)/2$ is even for even $n$ whenever $4 \mid n$. This would only be the case if $4 \mid n=2 d_0$, i.e. $2 \mid d_0$ (the case we excluded).
%\end{proof}

\begin{proposition}\label{intersection}
	The group $V'$ is a semidirect product $V'=C' \rtimes P'$.
	Moreover, the subgroup $H':=\langle h_0 \rangle \times \langle  p_k'^2 \mid k=1,\dots,t_l-1 \rangle$ is equal to the intersection $V' \cap H$.
\end{proposition}

\begin{proof}
	First consider $H\cap C'$. By construction $C'\lhd V'$ with $\rho(C')= C_{2d_0}^{a_l/2}$.
	By Lemma \ref{abelian}, it follows that $|C'|\leq 2(2d_0)^{a_l/2}$ and $h_0\in C'\cap H$.
	Hence $|C'|= 2(2d_0)^{a_l/2}$ and $H' \cap C'= \langle h_0 \rangle$.
	
	As $P'=\iota_2\circ \iota_1(\langle {{\bf n}_{\alpha_{2}}(1)},\dots, {{\bf n}_{\alpha_{a_l}}(1)}\rangle)$, it follows that $|P'|=|\mathfrak{S}_{t_l}| 2^{t_l-1}$.
	Moreover, $(p_i')^2\in H$ and $\rho(P')=\mathfrak{S}_{t_l}$.
	Hence $P'\cap H=\langle  p_k'^2 \mid k=1,\dots,t_l \rangle$. It is immediate that $H'\leq V'\cap H$ and so
	$$|V'|\geq |V'/V'\cap H||H'|=2^{t+1}|\mathfrak{S}_{t_l}||\langle (p_i')^2\mid 1\leq i\leq t_l\rangle|=|C'||P'|.$$
	Therefore $H'=V'\cap H$ and $C'\cap P'=1$.
	%The elements $p_k'$ satisfy the braid relations, by the lemma above. Moreover, as $(p_i')^2\in H$, 
	%More precisely, $(p_k p_{k+2})^2=h_k h_{k+1} h_{k+2} h_{k+3} \in H'$ since $p_k$ and $p_{k+2}$ commute.
	%there is a surjection $\mathfrak{S}_{a_l/2} \to  V'/C' H'$ and thus $|V'|\leq |H'| (2d_0)^{a_l/2} |\mathfrak{S}_{a_l/2}|$. On the other hand, the map $\rho':V' \to C_{2d_0} \wr S_{a_l/2}$ is surjective and $H'$ is in the kernel of it. Therefore, $H'=\mathrm{ker}(\rho')$.
\end{proof}

Note that for $i \geq 1$, the element $h_i h_{i+1}$ lies in the center of $L$ precisely when $i$ is odd. From this it follows that the intersection $H' \cap L \leq \Z(L)$ is contained in the center of $L$.

%\begin{lemma}\label{action}
%	Finally, we end these considerations by computing the action of $c_1 c_2$, $p_1 p_3$ on $h_1 h_2 h_3 h_4$. We obtain:
%	\begin{itemize}
%		\item Note that $c_i'H=(c_{2i-1}c_{2i})H$ and hence the formulas above for $h_i^{c_j}$ implies that $[C',H']=1$ %$(h_1 h_2 h_3 h_4)^{c_1 c_2}=h_0^2 h_1 h_2 h_3 h_4=h_1 h_2 h_3 h_4$ and more generally, $C$ centralizes $H'$.
%		\item The group $P'$ acts by permting the elements $\langle p_i' \rangle$, as usual {\color{blue} Do you just want the $(c_i')^{p_j'}$ here? I have written these above}.
%	\end{itemize}
%\end{lemma}

\section{Extension map for $L \lhd N$}

\subsection{A criterion for an extension map}

The following proposition, see \cite[Proposition 2.2]{TypD}, is useful for constructing extension maps.

\begin{proposition}\label{ext map}
Let $K \lhd M$ be finite groups, let the group
$E$ act on $M$, stabilizing $K$ and let $\mathbb{K} \subset \Irr(K)$ be $ME$-stable. Assume there exist
$E$-stable subgroups $K_0$ and $\hat{V}$ of $M$ such that
\begin{enumerate}[label=(\alph*)]
	\item the groups satisfy:
\begin{enumerate}[label=(\roman*)]
	\item   $K = K_0(K \cap \hat{V} )$ and $\hat{H} := K \cap \hat{V} \leq \Z(K)$,
\item  $M = K \hat{V}$ ;
\end{enumerate}
\item for $\mathbb{K}_0 :=
\cup_{\lambda \in \mathbb{K}}
\Irr(\lambda|_{K_0})$
there exist
\begin{enumerate}[label=(\roman*)]
\item a $\hat{V} E$-equivariant extension map $\Lambda_0$ with respect to $\hat{H} \lhd \hat{V}$ ; and
\item an $\varepsilon( \hat{V} )E$-equivariant extension map $\Lambda_\varepsilon$ with respect to $K_0 \lhd K_0 \rtimes \varepsilon(\hat{V})$ for
$\mathbb{K}_0$, where $\varepsilon : \hat{V} \to \hat{V} / \hat{H}$ denotes the canonical epimorphism.
\end{enumerate}
\end{enumerate}
Then there exists an $ME$-equivariant extension map with respect to $K \lhd M$ for $\mathbb{K}$.
\end{proposition}

We wish to apply this proposition in the following situation. We set $K_0:=[\Levi,\Levi]^F$ and we let $E:=\langle F_p \rangle$ be the group generated by field automorphisms in \cite[Definition 4.1]{TypeB}. Moreover, $K:=K_0 H'$, $\hat{V}:=V'$ and $M:=KV'=K_0 V'$. It follows from Proposition \ref{intersection} and the remarks following it that the group theoretic requirements in part (a) are satisfied. We will now work towards showing that the character theoretic properties of part (b) are also satisfied.
% Problem: Identify the action of $V'$ on $K_0$.
%Different idea: We pass to $\tilde{L}$ which has a structure of a direct product modulo $\Z(\tilde{G})$. We know that the characters in $b_L(\lambda)$ have stabilizer $L \Z(\tilde{L})$ in $\tilde{L}$. Then we construct everything in $\tilde{L}$ and go down. Diagonal automorphims are no problem since they get absorbed into $\tilde{L}$. 
%\end{remark}

The computations in the previous section allow us to conclude that assumption (b)(i) in Proposition \ref{ext map} is satisfied:

\begin{proposition}\label{ext H}
	There exists a $V'E$-equivariant extension map $\Lambda$ with respect to $H' \lhd V'$.
\end{proposition}

\begin{proof}
	Observe that $E$ centralizes $V'$. Hence, it is enough to construct for each character $\lambda \in \Irr(H')$ an extension to $V'_\lambda$.
	For a character $\lambda \in \Irr(H')$ its inertia group decomposes as $V'_\lambda=C' \rtimes P'_\lambda$ and $\Z(G)=H' \cap C'$. The character $\theta:=\lambda|_{\Z(G)}$ extends to a character of $C'$ with $\hat{\theta}(c_i')=\hat{\theta}(c_1')$ for all $i$. By the computation before Lemma \ref{abelian}, we see that the set $\{c_i' \mid i =1, \dots, a_l\}$ is stable under the group action of $P'=\langle p_i' \mid i=1,\dots, a_l-1 \rangle$. Therefore, the character $\hat{\theta}$ is $P'$-stable. Hence, there exists a unique character of $C' H'$ extending both $\lambda$ and $\hat{\theta}$. In particular, the so-obtained character is $P'_\lambda$-stable.
	
	According to Lemma \ref{homomorphism}, the map $\iota_2: \langle p_1,\dots, p_{t_l-1} \rangle \to P', p \mapsto p^{g_1} p^{g_2},$ is an isomorphism. Since $H_l=\langle h_0 \rangle (\langle p_1,\dots, p_{t_l-1}\rangle \cap H_l)$ and $h_0 \in \Z(G)$ this isomorphism can be extended to an isomorphism $\tilde{\iota}_2:H_l P \to H' P', p \mapsto p^{g_1} p^{g_2}$. According to \cite[Theorem 5.5]{TypeB}, there exists an extension map for $H_l \lhd V_l$ and this gives via the isomorphism $\tilde{\iota}_2$ an extension map from $H'$ to $P'$. Consequently, $\lambda$ extends to $P'_\lambda$.
	
	Using the $P'_\lambda$-stable extension of $\lambda$ to $C'$ and the extension of $\lambda$ to $P'_\lambda$ uniquely determines an extension to $V'_\lambda$.
\end{proof}

\subsection{Structure of $K_0$}
%We need the following lemma:

The  group $K_0=[\Levi,\Levi]^F$ is isomorphic to $\SL_2( \varepsilon q^{d_0})^{t_l} \times B_m(q)$, where $\varepsilon:=(-1)^{d+1}$.
%A concrete isomorphism is given by the mapping
%$$\Phi:[L,L] \to \SL_2(\varepsilon q^{d_0})^t \times B_m(q), N_{F^{d_0}/F}(x_{\alpha_1}(u)) \mapsto \begin{pmatrix}
%	1 & u \\ 0 & 1
%\end{pmatrix},$$
%and extending it $P'$-$S_{t}$-equivariantly. That is, for $p' \in P'$ inducing the permutation $\pi \in S_{t}$ we set $\Phi({}^{p'} N_{F^{d_0} /F}x_{\alpha_1}(u) ):=x_{\pi(1)}(u)$.
More concretely, set $L_1:=\langle  N_{F^{d_0}/F}(\x_{e_2-e_1}(1)),  N_{ F^{d_0}/F}(\x_{e_1-e_2}(1))\rangle$ and $L_i:=L_1^{p_1'\cdots p_{i-1}'}\cong {\rm SL}_2(\varepsilon q^{d_0})$. Here, $F=v_l F_q$ as in \ref{relative Weyl} and $N_{F^{d_0}/F}$ denotes the norm map. Then $K_0=L_1\times\dots \times L_{t_l}\times B_m(q)$.

\begin{lemma}\label{iso}
There exists a unique isomorphism
$$\Theta:L_1 \to \SL_2(\varepsilon q^{d_0})$$
with $\Theta(N_{F^{d_0}/F}(\x_{ \pm (e_1-e_2)}(u)))=\x_{ \pm (e_1-e_2)}(u)$.
\end{lemma}
\begin{proof}
	We first claim that $v_l^{d_0} \in \Z(\V_l)$. Firstly, $\rho(v_l)^{d_0} \in \Z(\W_l)$ and $v_l^{d_0}=v_{l,0}^{l}$ or $v_l^{d_0}=v_{l,0}^{2l}$. Hence, by \cite[Section 5.B]{TypeB}, we have $\mathbf{H}_l \leq \C_{\V_l}(v_l^{d_0})$. By \cite[Lemma 5.4]{TypeB} we deduce therefore that $v_l^{d_0} \in \Z(\V_l)$.
	
	Note that for $u \in \mathbb{F}$ we have ${}^{v_{l}^{d_0}} \x_{e_1-e_2}(u)=\x_{\varepsilon(e_1-e_2)}(\delta u)$ for some $\delta \in \{\pm 1\}$ by the Chevalley relations. Since $v_l^{d_0} \in \Z(\V_l)$ we have ${}^{v_{l}^{d_0}} \n_{e_1-e_2}(1)=\n_{e_1-e_2}(1)$. This implies $\delta=\varepsilon$ by the Chevalley relations \cite[Satz 2.1.6]{BS2006}. Therefore, $F^{d_0}(\x_{e_1-e_2}(u))=\x_{\varepsilon(e_1-e_2)}(\varepsilon u^{q^{d_0}})$. The claim of the lemma can now be deduced from this.
\end{proof}

	\begin{lemma}
		For $1\leq i\ne j\leq t_l$, the commutator $[L_i,c_j']=1$.
		Moreover, $[B_m(q),V']=1$.
	\end{lemma}
	\begin{proof}
		Denote by $\Phi_{A_{\mathcal{O}_{2i-1}}  \sqcup \mathcal{O}_{2i}}$ the type $A$ subsystem as in \cite[Notation 3.1]{BroughC}. By construction $L_i\leq \langle \X_\alpha\mid \alpha\in \Phi_{A_{\mathcal{O}_{2i-1}} \sqcup \mathcal{O}_{2i}}\rangle$ and $c_j'=x_{2j-1}x_{2j}$ for $x_i \in \langle \X_\alpha \mid \alpha \in  \Phi_{\mathcal{O}_i} \rangle$.
		Fix $\alpha\in \Phi_{A_{\mathcal{O}_{2i-1}}\sqcup \mathcal{O}_{2i}}$ and assume $i\ne j$.
		Then $\alpha\perp \Phi_{\mathcal{O}_i}$.
		By applying \cite[Remark 2.1.7]{BS2006}, it follows that $\x_\alpha(u)^{\n_\beta(1)}=\x_\alpha(u)$ for all $\beta \in \Phi_{ \mathcal{O}_i}$.
		Thus $[\x_\alpha(u),c_j']=1$ for $\alpha\in \Phi_{A_{\mathcal{O}_{2i-1}\sqcup \mathcal{O}_{2i}}}$ whenever $i\ne j$.
		
		The same argument also shows for $1\leq i\leq a_l$ that $[B_m(q),p_i]=1$ which implies $[B_m(q),P']=1$.
		It remains to consider $[B_m(q),c_1']$.
		Let $\alpha\in \Phi\cap \langle e_i\mid l+1\leq i\leq n\rangle$.
		Again applying \cite[Remark 2.1.7]{BS2006}, yields $\x_\alpha(u)^{c_1}=\x_{\alpha}(\varepsilon u)$ with $\varepsilon \in \{\pm 1 \}$. Hence as $[B_m(q),p_1]=1$, it follows that $\x_{\alpha}(u)^{c_1'}=\x_{\alpha}(u)^{c_1p_1^{-1}c_1p_1}=\x_\alpha(\varepsilon u)^{c_1p_1}=\x_{\alpha}(u)$. 
	\end{proof}
	
	It remains to understand the action of $c_i'$ on the factor $L_i$.
	Note that it suffices to explain the action of $c_1'$ on $L_1$ as $c_i'=(c_1')^{p_1'\cdots p_{i-1}'}$ and $L_i'=(L_1')^{p_1'\cdots p_{i-1}'}$.

%\begin{lemma}
%Let $\Phi_0:= \Phi \cap \langle e_1,e_2 \rangle$, $V_0:=\langle n_\alpha \mid \alpha \in \Phi_0 \rangle$ and $H_0=V_0 \cap H=\langle h_\alpha \mid \alpha \in \Phi_0$. Then $H_0$ acts trivially on $X_{e_1 -e_2}$ and $n_{e_1} n_{e_2}$ acts as $n_{e_1 -e_2}$ on $X_{e_1 -e_2}$.
%\end{lemma}
%
%\begin{proof}
%A computation in the Weyl group shows that $n_{e_1} n_{e_2} n_{e_1 - e_2} \in n_{e_1+ e_2} H_0$ and since $e_1+ e_2 \perp e_1 -e_2$ and both are long roots it follows that $n_{e_1+e_2}$ centralizes $X_{e_1-e_2}$. Hence, to show the second claim it suffices to proof the first one, i.e. that $H_0=\langle h_0, h_{e_1 -e_2} \rangle$ centralizes $X_{e_1-e_2}$. This is however again clear by the Chevalley relations, since $A_{\alpha,\alpha}=2$ for all roots $\alpha$. 
%\end{proof}
%
%Generalizing the previous lemma, we obtain:

\begin{lemma}\label{graph}
	The element $c_1'$ acts as $v'_l$ on $L_1$ and $c_1'$ acts trivially on $B_m(q)$.
\end{lemma}

%\begin{lemma}\label{graph}
%	The element $c_1'^{d_0}$ acts as transpose inverse on $L_1 \cong\SL_2( \varepsilon q^{d_0})$ and $c_1'$ acts trivially on $B_m(q)$.
%\end{lemma}

\begin{proof}
We consider the element $x:=v'_l (c_1' \cdots c'_{t_l})^{-1}$. A computation in the Weyl group shows that $x \in H_l$ and we claim that the image of $x$ in $H_l/\langle h_0 \rangle$ gets centralized by $P:=\langle p_1,\dots,p_{a_l-1} \rangle$. By \cite[Section 5.B]{TypeB} the group $C/\langle h_0 \rangle$, where $C:=\langle c_1,\dots c_{a_l} \rangle$, embedds into $W_l$.
% this is because $c_1^{2d_0} \in \langle h_0 \rangle and $[c_i,c_j]=h_0$.
In particular, since $P$ centralizes the image of $c_1' \cdots c'_{t_l}$ in $W_l$ it follows that $P$ centralizes the image of $x$ in $H_l/\langle h_0 \rangle$. Using \cite[Section 5.B]{TypeB}, which describes the action of $P$ on $H_l$, it follows that $x \in \langle h_0, h_1 \cdots h_{a_l} \rangle \cap H_l$. By the Chevalley relations we conclude that $x$ centralizes $\X_{\pm (e_1-e_2)}$. In particular, since $c_2',\dots, c'_{t_l}$ centralize $\X_{\pm (e_1-e_2)}$ as well it follows that $c_1'$ acts as $v_l'$ on $L_1$.
%By the Chevalley relations $[\n_{e_1+e_2}(1),\n_{e_1-e_2}(1)]=1$ and $[\n_{e_1 \pm e_2}(1),{}^{v_l} \n_{e_1 \pm e_2}(1)]=1$.
%
%
%
%% element $v_l$ has order $d$ in the Weyl group and so $v_l^{d}=v_{0,l}^{2l} \in \Z(\V_l)$.
% Since ${}^{v_l^{d_0}} \n_{e_1-e_2}(1)=\n_{e_1+e_2}(1) \h_{e_1}(\pm 1)$, it follows that the element $n:=\n_{e_1+e_2}(1) \n_{e_1-e_2}(1)$ is $v_l^{d_0}$-stable. We conclude from this that the element $\prod_{i=0}^{d_0-1} {}^{v_l^i} n$ is $v_l$-stable and has the same image in the Weyl group as $c_1'^{d_0}$. Hence, both elements which lie in $V_{\mathcal{O}_1 \cup \mathcal{O}_2}$ differ by an element in $\C_{H_{\mathcal{O}_1 \cup \mathcal{O}_2}}(v_l)=\langle h_0,h_1 h_2 \rangle$.
%%(use Späth's equality (?) ).
%However, $h_1 h_2$ acts trivially on $\X_{ \pm (e_1 -e_2)}$ by construction and the Chevalley relations. Hence, $\prod_{i=0}^{d_0-1} {}^{v_l^i} n$ and $c_1'^{d_0}$ act in the same way on $\X_{ \pm (e_1-e_2)}$. By the Chevalley relations $n$ acts as transpose inverse on $\X_{\pm (e_1-e_2)}$. The first claim of the lemma follows from this. The second claim of the lemma follows from the fact that the image of $c_1' \in \V_{\mathcal{O}_{1} \cup \mathcal{O}_{2}}$ in the Weyl group lies in a type $D$ Weyl group and using \cite[Bemerkung 2.1.7]{BS2006}.
%%as its projection onto $H_{\{1,2 \} }$ is $h_{e_1}(\omega) h_{e_2}(\omega)$.
\end{proof}

\begin{lemma}\label{ext commutator}
	 There exists an $N E$-stable $\tilde{L}$-transversal $\mathbb{K}_0 \subset \Irr(L_0)$. Moreover, for the constructed set $\mathbb{K}_0$, there exists an $V'/H' \langle F_p \rangle$-equivariant extension map $L_0 \lhd L_0 \rtimes V'/H'$.
\end{lemma}
\begin{proof}
 Observe that $\mathfrak{S}_{t_l} \cong P'/H' \cap P'$ permutes the set $\{ c_i' \mid i=1, \dots t_l \}$. That is we have $$L_0 \rtimes V'/H' \cong (L_1 \rtimes \langle c_1' H' \rangle) \wr \mathfrak{S}_{t_l} \times B_m(q).$$
The field automorphism $F_p \in E$ acts diagonally on all factors in the decomposition of $L_0$.
 Observe that the group $\tilde{L}$ induces all diagonal automorphisms on $L_0$. We can therefore choose a $\tilde{L}$-transversal $\mathbb{K}_0 \subset \Irr(L_0)$ such that every $\chi \in \mathbb{K}_0$ is of the form $\chi=\chi_1 \times \dots \times \chi_{t_l} \times \chi'$ with $\chi_i \in \Irr(\SL_2(\varepsilon q^{d_0}))$, $\chi' \in \Irr(B_m(q))$, such that $\chi_i=\chi_j$ are either equal or not in the same $\GL_2(\varepsilon q^{d_0})$-orbit. 
 
 Note that $c_1' H'$ has order $2d_0$ and acts as $v_l'$ on $L_1$ by Lemma \ref{graph}.
% Since $c_1'^{d_0}$ acts on $\SL_2(\varepsilon q^{d_0})$ as transpose inverse 
If $d$ is odd then $v_l'$ is $2d$-regular. The proof of Lemma \ref{iso} therefore shows that ${}^{v_l'^d} \x_{e_1-e_2}(u)=\x_{-(e_1-e_2)}(-u)$ when $d$ is odd. By Lemma \ref{graph} and the construction of the isomorphism $\Theta: L_1 \to \SL_2(\varepsilon q^{d_0})$ in Lemma \ref{iso}, the automorphism $c_1'$ therefore acts as $F_q$ or as $\tau F_{q}$, where $\tau$ is tranpose-inverse, on $\SL_2(\varepsilon q^{d_0})$. Note that transpose-inverse is an inner automorphism of $\SL_2(\varepsilon q^{d_0})$.

%  combination of a field automorphism and an inner automorphism. In other words, $\mathrm{ad}(c_1')=F_q \mathrm{ad}(l)$ with $l \in \SL_2(\varepsilon q^{d_0})$ as automorphisms on $\SL_2(\varepsilon q^{d_0})$.
 The characters $\chi_i, \chi'$ satisfy $A'(\infty)$ by \cite{TypD}. In particular, the set $\mathbb{K}_0$ is $NE$-stable. Thus, there exists an $\langle c_1,F_p \rangle$-equivariant extension map for $L_1 \lhd L_1 \rtimes \langle c_1' H' \rangle$. The claim follows now from applying \cite[Lemma 3.6]{TypD}.
% case X=SL_2(q^{d_0}), Y=\langle c_1' \rangle, A=
\end{proof}

We define $\mathbb{K}:=\Irr(K \mid \mathbb{K}_0)$ and $\mathbb{T}:=\Irr(L \mid  \mathbb{K} )$. Since $\tilde{L}/L_0$ is abelian, the sets $\mathbb{K}$ and $\mathbb{T}$ are again $N E$-stable $\tilde{L}$-transversals of $\Irr(K)$ resp. $\Irr(L)$.

Note that $K=K_0 H'=L_0 H'$ is a central product. Hence, $\mathbb{K}$ consists of all extensions to $K$ of characters in $\mathbb{K}_0$.

\begin{proposition}
	There exists an $ME$-equivariant extension map with respect to $K \lhd M$ for $\mathbb{K}$.
\end{proposition}

\begin{proof}
We observe that by Lemma \ref{ext commutator} and Lemma \ref{ext H} the assumptions of Proposition \ref{ext map}(b) are satisfied.
\end{proof}

\begin{proposition}\label{ext}
%\begin{enumerate}[label=(\alph*)]
%	\item The set $\mathbb{T} \subset \Irr(L)$ is an $NE$-stable $\tilde{L}$-transversal.
There exists an $NE$-equivariant extension map from $L$ to $N$ for $\mathbb{T}$.
%\end{enumerate}
\end{proposition}

\begin{proof}
The proof is similar to \cite[p.28]{TypD} using the properties we proved so far.
For the proof it is sufficient to construct for
every $\theta \in \mathbb{T} = \Irr(L \mid \mathbb{K})$ some $(NE)_\theta$-stable extension of $\theta$ to $N_\theta$. A character $\theta \in \mathbb{T}$
lies above a unique $\theta_0 \in \mathbb{K} = \Irr(K \mid \mathbb{K}_0)$. Moreover some extension $\tilde{\theta}_0 \in \Irr(L_{\theta_0}
)$ to
$L_{\theta_0}$
satisfies $\Ind_{L_{\theta_0}}^{L}(\tilde{\theta}_0
) = \theta$. By the properties of $\mathbb{K}$ we see $N_{\theta_0} = L_{\theta_0} M_{\theta_0}$.
%
%that is $(N E)_{\tilde{\theta}_0}$-stable.

By Proposition \ref{ext map} the character $\theta_0$ has a $(V'E)_{\theta_0}$-stable extension to $M_{\theta_0}$. According to \cite[Lemma 4.1]{BS} this defines an extension $\phi$ of $\tilde{\theta}_0$
to $N_{\tilde{\theta}_0}$
since $N_{\tilde{\theta}_0}\leq L_{\theta_0} M_{\theta_0}$. By
the construction we see that $\Ind^{N_\theta}_{N_{\tilde{\theta}_0}}(\phi)$ is an extension of $\theta$.
As $\mathbb{T}$ is an $M$-stable $\tilde{L}$-transversal $\tilde{N}_{\theta_0} = \tilde{L}_{\theta_0} M_{\theta_0}$
and $(\tilde{N} E)_{\theta_0} = \tilde{L}_{\theta_0} (ME)_{\theta_0}$. Hence
this extension of $\theta_0$ defines an extension of $\theta$ as required.
\end{proof}

In our situation, it is easily possible to extend the extension map from Proposition \ref{ext} to an $NE$-equivariant extension map on $\Irr(L)$:

\begin{corollary}\label{ext2}
	There exists an $NE$-equivariant extension map from $L$ to $N$.
\end{corollary}

\begin{proof}
Let $\Lambda$ be the extension map for $\mathbb{T}$ from Proposition \ref{ext}. We extend $\Lambda$ to $\Irr(L)$ by extending it $\tilde{L}$-equivariantly, i.e. we define $\Lambda({}^{\tilde{l}} \theta):={}^{\tilde{l}} \Lambda(\theta)$ for $\theta \in \mathbb{T}$ and $\tilde{l} \in \tilde{L}$. Note that this is well-defined (i.e. does not depend on the choice of $\tilde{l}$) since $\tilde{L}_\lambda=\tilde{L}_{\Lambda(\lambda)}$ by Remark \ref{rmk}. It is now easily checked that the extended map is still $NE$-equivariant.
\end{proof}

%From the proof of Proposition \ref{ext} we obtain an alternative proof that the $d$-cuspidal pair $(\Levi,\lambda)$ satisfies the conclusion of Corollary \ref{coro}(a)+(b).

\begin{remark}
	The conclusion of Proposition~\ref{ext} will also hold for any $d$-split Levi subgroup $\bL$ with root system $\Phi_{\bL}\cong \Phi_{B_{l}}\sqcup_{i}\Phi_{A_{2i-1}}^{s_i}$.
	In particular, for suitably chosen elements $g_1,\dots ,g_{2i}$ a corresponding map $\iota_2$ for the factor $\Phi_{A_{2i-1}}^{s_i}$ can be constructed (as in \cite[Section 4.C.2]{BroughC} for the analogous situation in type C).
\end{remark}

\subsection{Verifying the inductive condition}

Before proving our main theorem, we show how one obtains Assumption \ref{assumption} for cases 1 and 3 of Table \ref{table} from \cite{TypeB}:

\begin{remark}\label{modifaction}
	We wish to construct a $d$-split Levi subgroup of $B_n(q)$ of rational type $B_m(q) (q^{d_0}+ (-1)^{d+1} )^a$ for some integers $m$ and $a$. For this we modify the construction in \ref{Sylow} by defining $n':=n-m$ and by keeping the same definitions as there. It follows that the Levi subgroup $\Levi$ with root system $\Phi_2=\Phi \cap \langle e_i \mid n' \leq i \leq n \rangle$ is $vF_q$-stable (where $v=({\bf n}_{\alpha_1}(1)\cdots {\bf n}_{\alpha_{n'}}(1))^{\frac{2n'}{d_0}}$) and $\Levi^{vF_q}$ has type $B_m(q) (q^{d_0}+ (-1)^{d+1} )^a$. One observes that the considerations in \cite[Section 5.D]{TypeB} apply verbatim to our situation. In particular, Assumption \ref{assumption} follows from the proof of \cite[Theorem 4.2]{TypeB}. 
\end{remark}

	\begin{theorem}
	Let $G$ be a quasi-simple group of Lie type $B_n$ or $C_n$ defined over the finite field $\mathbb{F}_q$ for $q$ a prime power of an odd prime and let $\ell \geq 5$ not dividing $q$.
	Then every $\ell$-block of $G$ satisfies the iAM-condition.
\end{theorem}

\begin{proof}
	According to the reduction theorem in \cite[Theorem 12.6]{Jordan3} in order to show the theorem it suffices to check the iAM-condition for isolated blocks relative to the normalizer of their Cabanes subgroup. As explained in the proof there, we may also assume that $G$ has non-exceptional Schur multiplier. Let $b$ be a isolated block associated to the $d$-cuspidal pair $(\Levi,\lambda)$. According to Theorem \ref{bijection} it suffices for this to check that Assumption \ref{assumption} is satisfied. For groups of type $C_n$ this criterion was verified in \cite[Theorem 1.2]{BroughC}. If $\G$ is of type $B_n$, then $\Levi$ is one of the Levi subgroups of Table \ref{table}. For the Levi subgroups in cases 1 and 3 Assumption \ref{assumption} follows from Remark \ref{modifaction}. For the Levi subgroups in case 2 the Assumption \ref{assumption}(i) follows from Corollary \ref{ext2}. Assumption \ref{assumption}(ii) on the other hand follows from Remark \ref{rem}(i) and Remark \ref{rmk}.
\end{proof}

\end{document}